\documentclass[11pt]{article}

\usepackage{mathtools}
\usepackage{amsmath, amsthm, amsfonts,amssymb}
\usepackage{enumitem}
\usepackage{graphicx}
\usepackage{colortbl}
\usepackage{tikz}
\usepackage[utf8]{inputenc}
\usepackage{esint}
\usepackage{mathrsfs}
\usepackage[T1]{fontenc}
\usepackage{mathrsfs}  
\usepackage{varwidth}
\usepackage{bbm} 
\usepackage{enumitem}
\usepackage{cite}
\usepackage{mathtools}
\usepackage{array}
\usepackage{graphicx}
\usepackage{caption}
\usepackage{subcaption}
\usepackage{dsfont}
\usepackage{ushort}
\usepackage{accents}
\usepackage{booktabs}
\usepackage{nicefrac}

\usepackage{tikz}
\usetikzlibrary{angles, quotes,calc,patterns}
\usepackage{tikz}
\usepackage{tikz-3dplot}
\usepackage{lineno} % to show line numbers
\usepackage{mathtools}
\mathtoolsset{showonlyrefs}

\numberwithin{equation}{section}

\usepackage{pgfplots}
\pgfplotsset{compat = newest}

\usepackage[titletoc, toc, page]{appendix} % to refer to appendices as letters

\newcommand{\intx}{\int_{\mathbb{R}^n}}
\newcommand{\ints}{\int_{\mathbb{S}^{n-1}}}
\newcommand{\lp}{\left(}
\newcommand{\rp}{\right)}
\newcommand{\eps}{\varepsilon}
\newcommand{\pp}{P_{\Omega^\perp}}

\newcommand\blfootnote[1]{%
	\begingroup
	\renewcommand\thefootnote{}\footnote{#1}%
	\addtocounter{footnote}{-1}%
	\endgroup
}

\def\d{\,\mathrm{d}}

\def\Id{\mathrm{Id}}

\usepackage[colorlinks=true,linkcolor=blue,citecolor=blue,urlcolor=blue,breaklinks]{hyperref}

\usepackage{hyperref}
\usepackage{url}
\newcommand{\red}[1]{{\textcolor{red}{#1}}}

\def\R{\mathbb{R}}

\def\S{\mathbb{S}}

\def\1{\mathds{1}}

\let\mc=\mathcal

\def\d{\,\mathrm{d}}

\def\p{\,\partial}
\def\O{\Omega}

\def\udoto{u\cdot\Omega}

\def\heo{h_\eta(u\cdot\Omega)}

\def\ga{G_{\eta A_\Omega}}
\def\etarho{\eta(\rho)}
\def\proopp{P_{\Omega^\perp}}

\def\upp{u_\perp}
\def\upar{u_\parallel}

\let\eps\varepsilon

\newcommand{\GetaA}{G_{\eta(\rho) A_{\Omega}}}
\newcommand{\Getatheta}{g_{\eta}(\theta) }

\newcommand{\GCI}{\boldsymbol{\psi}(u)}

\newtheorem{thm}{Theorem}[section]
\newtheorem{cor}[thm]{Corollary}
\newtheorem{lem}[thm]{Lemma}
\newtheorem{prp}[thm]{Proposition}

\newtheorem{asm}[thm]{Assumption}
\theoremstyle{definition}

\newtheorem{dfn}[thm]{Definition}
\theoremstyle{remark}
\newtheorem{remark}[thm]{Remark}

\hypersetup{pdftitle={Anisotropic}}
\hypersetup{pdfauthor={Merino, Plunder, Wytrzens  \& Yoldaş }}

\author{Sara Merino-Aceituno\footnote{Faculty of Mathematics, University of Vienna, Oskar-Morgenstern-Platz 1, 1090 Vienna, Austria. \\ sara.merino@univie.ac.at \& claudia.wytrzens@univie.ac.at}
\and Steffen Plunder\footnote{Institute for the Advanced Study of Human Biology (ASHBi), KUIAS, Kyoto University, Faculty of Medicine Bldg. B, Kyoto, 606-8303, Japan. plunder.steffen.2a@kyoto-u.ac.jp}
\and Claudia Wytrzens{$^*$}
\and Havva Yolda\c{s}\footnote{Delft Institute of Applied Mathematics, Faculty of Electrical Engineering, Mathematics and Computer Science, Delft University of Technology, Mekelweg 4, 2628CD Delft, The Netherlands. h.yoldas@tudelft.nl}
\footnote{Theoretical Sciences Visiting Program (TSVP), Okinawa Institute of Science and Technology Graduate University, Onna, 904-0495, Japan. havva-yoldas@oist.jp}}

\title{Macroscopic effects of an anisotropic Gaussian-type repulsive potential: nematic alignment and spatial effects}

\makeindex

\begin{document}
	
	\maketitle
	
	\vspace{-10pt}

	\begin{abstract}
		\noindent 

 Elongated particles in dense systems often exhibit alignment due to volume exclusion interactions, leading to packing configurations. Traditional models of collective dynamics typically impose this alignment phenomenologically, neglecting the influence of volume exclusion on particle positions. In this paper, we derive nematic alignment from an anisotropic repulsive potential, focusing on a Gaussian-type potential and first-order dynamics for the particles. By analyzing larger particle systems and performing a hydrodynamic limit, we study the effects of anisotropy on both particle density and direction. We find that while particle density evolves independently of direction, anisotropy slows down nonlinear diffusion. The direction dynamics are affected by the particles' position and involve complex transport and diffusion processes, with different behaviors for oblate and prolate particles. The key to obtaining these results lies in recent advancements in Generalized Collision Invariants offered by Degond, Frouvelle and Liu (KRM 2022) in \cite{DFL22}.
		
	\end{abstract}
	
	\tableofcontents
	
	\section{Introduction} \label{sec:intro}

Volume exclusion interactions play a central role in many physical and biological systems. In particular, they are fundamental in explaining emergent patterns like swarming \cite{WL12}, and spontaneous alignment of anisotropic particles \cite{DB15}. The latter is called \emph{nematic alignment}. The term \emph{nematic} indicates that the alignment takes place in a given direction (not necessarily in a given orientation): if $u_1,\, u_2\in \R^n$ are such that $|u_1|=|u_2|=1$, we say that the two vectors are \emph{nematically} aligned if $u_1=\pm u_2$.  Nematic alignment is sometimes referred to as \emph{apolar} alignment since it is in contrast to \emph{polar} alignment which requires $u_1=u_2$. Some examples of nematic alignment can be found in suspensions of rod-like particles in high-densities \cite{BM08}, and biological systems like myxobacteria \cite{DMY18}.

\subsection{Volume exclusion and nematic alignment}
\label{ssec:volume_exclusion}

Various volume exclusion models have been proposed to investigate cell dynamics, such as the vertex model \cite{AGS17, FOBS14,HN15} and other packing systems \cite{DFM17}. However, most agent-based models consider the agents as point-particles and impose phenomenological behavior that is assumed to be caused by volume exclusion interactions. Particularly, in most of these mathematical models for collective dynamics, the particle alignment is imposed via a force term without dealing with the contact interactions directly, see, e.g., \cite{DMY17,DMY18,DM08}.

Typical models for collective dynamics with nematic alignment take the following shape:  agents move at a constant speed and try to align their direction of motion with one of their neighbors up to some noise. Specifically, we consider $N$ agents who are identified by their positions $X_i \in \mathbb{R}^n, \, n \in \{2,3\}$ and their directions $U_i \in \mathbb{S}^{n-1}$ on the $(n-1)$-dimensional sphere. Then their dynamics are governed by 
\begin{subequations} 
\begin{align}  
		\d X_i &=  v_0  U_i \d t, \label{eq:nematic_X} 
	\\	\d U_i &= \frac{1}{N}\sum_{j=1}^N K(|X_i-X_j|) \nabla_{U_i}  V_{\text{nem}}(U_i, U_j) \d t + P_{U_i^\perp}\circ\sqrt{2D_u}\d B_i,  \label{eq:nematic_u}
	\\		V_{\text{nem}}(U_i, U_j)&:=\lambda (U_i\cdot U_j)^2, \label{eq:nematic_pot}
\end{align} 
\end{subequations} 
where $\nabla_{U_i}$ is the gradient on the sphere, $V_{\text{nem}}(U_i,U_j)$ is the potential producing nematic alignment, $P_{U_i^\perp}$ is the orthonormal projection onto the orthonormal space to $U_i \in \S^{n-1}$, and $(B_i)_{i=1,\dots,N}$ are independent Brownian motions for $i=1,\dots,N$. The stochastic differential equation \eqref{eq:nematic_u} must be understood in the Stratonovich sense. This is indicated with the symbol `$\circ$' and it ensures that $U_i$ remains on the sphere for all times where the solution is defined. The constant $v_0$ in \eqref{eq:nematic_X} is the speed of the particles and $D_u>0$ in \eqref{eq:nematic_u} is the diffusion constant of the directions.  Moreover, the function $K\geq 0$ is an interaction kernel measuring the influence of the potential force depending on the distance between particles. The constant $\lambda>0$ describes the strength of the alignment force, which is expressed as the gradient flow dynamics of the potential $V_{\text{nem}}=V_{\text{nem}}(U_i, U_j)$. One can easily check that, indeed, the maximizer of this potential corresponds to $U_i=\pm U_j$, i.e., when two particles are \emph{nematically} aligned. In this regard, we say that the model \eqref{eq:nematic_X}-\eqref{eq:nematic_pot} imposes alignment for the particles. We refer the reader to, e.g., \cite{DFM-A17, DMY17,DMY18, DM-A20, HMS18, NGC12}, for models that use this approach or a similar one.

In the present work, we follow a different approach. We do not wish to impose alignment directly, but to investigate how it might emerge naturally from volume exclusion interactions. In particular, we study how volume exclusion interactions affect \textit{both} the directions and the positions of the particles. 

However, deriving continuum equations from agent-based models that undergo contact interactions is mathematically extremely challenging, see, e.g., \cite{BBRW17,BCS23}. This is why up-to-date there is no rigorous coarse-graining for excluded volume dynamics starting from the first principles. Even in the widely-studied Boltzmann equation, the derivation from discrete dynamics (a particle system undergoing elastic collisions) is still unknown for large times \cite{BGS-RS22}. For this reason, contact interactions are often approximated by using soft interaction potentials, like repulsive potentials, that produce active forces when two agents get closer than a given distance \cite{BV13,CCWWZ20,DM-AVY19}.

In this article, we focus on a particular repulsive potential that is used for simulating the interactions of anisotropic particles: the Gaussian repulsive potential \cite{BP72} for elliptic (dimension $n=2$) or spheroidal (dimension $n=3$) particles, see Figure \ref{fig:spheroids}. For related works in the biophysics literature see, e.g., \cite{BM08,BM08-2}.

Our goal is to investigate for which shape of the Gaussian potential we obtain an alignment force for the continuum equation, and what is the effect of this force on the positions of the particles. The reason for the last point is that in classical models for nematic alignment, the potential only modifies the direction of motion of the agents, but not their positions. One can expect that, from interaction forces, agents may \emph{push }each other modifying their positions. The question of interest here is what effect this \emph{pushing} has on the positions of the particles. 
 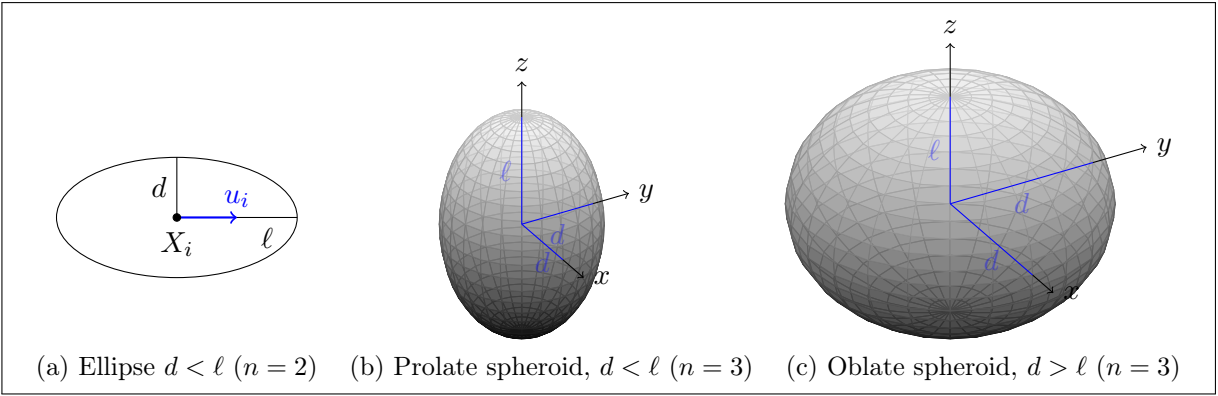
\begin{figure}[ht!]
	\begin{center}
		\fbox{ 
			\begin{tabular}[c]{c c c}
				 \begin{tikzpicture}[dot/.style={draw,fill,circle,inner sep=1pt}]
				 	\draw (0,0) -- (1.6,0) node [near end,below] {$\ell$};  
				 	\draw (0,0) -- (0,0.8) node [midway,left] {$d$};
				 	\draw[thick,blue,->] (0,0) -- (0.8,0) node[above] {$U_i$};
				 	\node[dot,label={below:$X_i$}] (Xi) at (0,0) {};
				 	\draw (0,0) ellipse (1.6 and 0.8);
			{\color{white}
                \begin{varwidth}{4cm}
						\draw[thick,->] (0,-1.6) -- (0.8,-1.6) node[above] {};
				\end{varwidth}
    }
				 \end{tikzpicture} & 
				 \begin{tikzpicture}
				 	\pgfmathsetmacro{\p}{1}
				 	\pgfmathsetmacro{\q}{1.5}
				 	\begin{axis}[scale=1.5,
				 		xlabel = {$x$},
				 		ylabel = {$y$},
				 		zlabel = {$z$},
				 		view = {60}{30},
				 		domain = 0 : pi,
				 		y domain = 0 : 2 * pi,
				 		z buffer = sort,
				 		unit vector ratio = 1 1,
				 		hide axis,
				 		colormap/blackwhite,
				 		declare function = {
				 			xp(\x, \y) = \p * sin(deg(\x)) * cos(deg(\y));
				 			yp(\x, \y) = \p * sin(deg(\x)) * sin(deg(\y));
				 			zp(\x, \y) = \q * cos(deg(\x));
				 		}, ]
				 		\addplot3[patch,fill opacity=0.7]({xp(x, y)}, {yp(x, y)}, {zp(x, y)});
				 		\draw[-,blue,fill opacity=0.4] (0, 0, 0) -- (\p, 0, 0) node[midway,below]{$d$};
				 		\draw[-,blue,fill opacity=0.4] (0, 0, 0) -- (0, \p, 0) node[midway,below]{$d$};
				 		\draw[-,blue,fill opacity=0.4] (0, 0, 0) -- (0, 0, \q) node[midway,left]{$\ell$};
						
				 		\draw[->] (\p, 0, 0) -- (\p+0.5, 0, 0) node[right]{$x$};
				 		\draw[->] (0, \p, 0) -- (0, \p+0.5, 0) node[right]{$y$};
				 		\draw[->] (0, 0, \q) -- (0, 0, \q+0.5) node[above]{$z$};
				 	\end{axis}
				 \end{tikzpicture} & 
				 \begin{tikzpicture}
				 	\pgfmathsetmacro{\p}{2.0}
				 	\pgfmathsetmacro{\q}{1.5}
				 	\begin{axis}[scale=1.5,
				 		xlabel = {$x$},
				 		ylabel = {$y$},
				 		zlabel = {$z$},
				 		view = {60}{30},
				 		domain = 0 : pi,
				 		y domain = 0 : 2 * pi,
				 		z buffer = sort,
				 		unit vector ratio = 1 1,
				 		hide axis,
				 		colormap/blackwhite,
				 		declare function = {
				 			xp(\x, \y) = \p * sin(deg(\x)) * cos(deg(\y));
				 			yp(\x, \y) = \p * sin(deg(\x)) * sin(deg(\y));
				 			zp(\x, \y) = \q * cos(deg(\x));
				 		}, ]
				 		\addplot3[patch, fill opacity=0.7]({xp(x, y)}, {yp(x, y)}, {zp(x, y)});
				 		\draw[-,blue,fill opacity=0.4] (0, 0, 0) -- (\p, 0, 0) node[midway,below]{$d$};
				 		\draw[-,blue,fill opacity=0.4] (0, 0, 0) -- (0, \p, 0) node[midway,below]{$d$};
				 		\draw[-,blue,fill opacity=0.4] (0, 0, 0) -- (0, 0, \q) node[midway,left]{$\ell$};

				 		\draw[->] (\p, 0, 0) -- (\p+0.5, 0, 0) node[right]{$x$};
				 		\draw[->] (0, \p, 0) -- (0, \p+0.75, 0) node[right]{$y$};
				 		\draw[->] (0, 0, \q) -- (0, 0, \q+0.75) node[above]{$z$};
				 	\end{axis}
				 \end{tikzpicture} \\
				\small (a) Ellipse $d<\ell$ ($n=2$) & 
				\small (b) Prolate spheroid, $d<\ell$ ($n=3$) &
				\small (c) Oblate spheroid, $d>\ell$ ($n=3$) 
		\end{tabular} }
	\end{center}
	\vspace{-0.5cm}
	\caption{\small Spheroids are obtained by rotating an ellipse, shown in (a), around one of its principal axes. If the revolution is around the major axis, the spheroid is called prolate (b); if it is around the minor axis, it is called oblate (c).}
	\label{fig:spheroids} 
\end{figure}

\subsection{A discrete model for anisotropic particles} \label{sec:intro_disc_mod}

In this article, we study particles with either elliptic shape in the case of dimension $n=2$ or spheroidal shape for dimension $n=3$ (see Figure \ref{fig:spheroids}). 
In both cases, the particles are identical and are identified by their center $X\in \mathbb{R}^n$, the direction of one of the axes specified by a unit vector $U\in \mathbb{S}^{n-1}$, and the lengths of the major and minor axes. In dimension $n=2$, $U$ denotes the direction of the major axis whose length is $\ell \geq 0$ and the length of the minor axis is denoted by $d \geq 0$. In dimension $n=2$, the \emph{main axis} is the principal axis, i.e., the axis with longer length; whereas in dimension $n=3$, the main axis is the axis of rotation.

We define the constant $\chi$ to characterize the shape of ellipses and spheroids,
\begin{align} \label{eq:def-chi}
	\chi :=\frac{\ell^2-d^2}{\ell^2+d^2},
\end{align} which measures the anisotropy of particles.
 In dimension $n=2$, we have that $\chi\in[0,1]$. In this case, $\chi=0$ and $\chi=1$ correspond to circular and rod-shaped particles, respectively.  In dimension $n=3$, we have $\chi \in [-1,1]$. In this case, the negative values of $\chi$ correspond to oblate particles (rotation around the minor axis), and the positive values of $\chi$ correspond to prolate particles (rotation around the major axis); see Figure \ref{fig:spheroids}. In particular, $\chi=0$ corresponds to spheres; $\chi = -1$ to infinitely flat disks; and $\chi = 1$ to infinitely thin rods.
 
 We consider $N$ identical particles identified by their centers $X_i\in\R^n$ and the direction of their main axes $U_i\in \mathbb{S}^{n-1}$ (notice that this is not uniquely defined as $U_i$ and $-U_i$ prescribe the same direction) for $i=1, \dots, N$. Two particles $(X_i,U_i)$, $(X_j,U_j)$ are said to be (nematically) aligned when $U_i=\pm U_j$.
 For simplicity, here we consider only inert particles, i.e, $v_0=0$. However, this could be easily extended. 
 
We assume that the repulsive potential $V_b$ acts on the distance between the centers of two particles $X_j-X_i$ and the directions of their main axes $U_i,U_j$, i.e., $V_b(U_i,U_j, X_j-X_i)$. The model follows the steepest gradient descent of the potential $V_b$ together with some noise both in the positions of the centers and in the directions of the main axes and it is given by
\begin{subequations}   \label{eq:discrete_system}
\begin{align}
		\d X_i &=-\mu \frac{1}{N}\sum_{j=1}^N \nabla_{X_i}V_b(U_i, U_j, X_j-X_i) \d t+ \sqrt{2D_x}\d B_i, \label{eq:discrete_system_X}\\ 
		\d U_i &=-\lambda \frac{1}{N}\sum_{j=1}^N \nabla_{U_i}V_b(U_i, U_j, X_j-X_i) \d t+P_{U_i^\perp}\circ \sqrt{2D_u}\d \tilde{B}_i,  \label{eq:discrete_system_u}
\end{align} 
	\end{subequations} where $\nabla_{U_i}$ and $P_{U_i^\perp}$ are the same as before, $B_i, \tilde B_i$ are independent Brownian motions for $i=1,\dots,N$; and $\mu, \lambda, D_x, D_u$ are positive constants. The potential $V_b$ corresponds to an anisotropic repulsive potential. In \cite{BKMT16}, the authors propose a model to describe the motion of spheroidal particles that are suspended in an incompressible fluid. The system \eqref{eq:discrete_system_X}-\eqref{eq:discrete_system_u} can be seen as the overdamped regime for these equations when there is no fluid. 

\begin{figure}[ht!] 
	\centering
	\fbox{ 
		\begin{tabular}[c]{m{15em} m{12em}}
			\begin{tikzpicture}[dot/.style={draw,fill,circle,inner sep=1pt}]
				{\color{white}
                \begin{varwidth}{4cm}
					\draw[thick,->] (0,2) -- (0.8,2) node[above] {};
				\end{varwidth}
    }
				\draw[rotate=90] (0,0) -- (2,0) node [near end,left] {$\ell$};  
				\draw[rotate=90] (0,0) -- (0,1) node [near end,below] {$d$};
				\node[dot,label={below:$x_1$}] (X1) at (0,0) {};
				\draw[thick,blue,->,rotate=90] (0,0) -- (1,0) node[midway, right] {$u_1$};
				\draw[fill=black!30!green, fill opacity=0.1,rotate=90] (0,0) ellipse (2 and 1);
				\draw[shift={(-0.17 cm,-3.75 cm)},rotate=60] (5,0) -- (7,0) node [near end,left] {$\ell$};  
				\draw[shift={(-0.17 cm,-3.75 cm)},rotate=60] (5,0) -- (5,1) node [midway,above] {$d$};
				\node[dot,label={below:$x_2$}] (X2) at (2.325,0.59) {};
				\draw[thick,blue,->,shift={(-0.17 cm,-3.75 cm)},rotate=60] (5,0) -- (6,0) node[midway, right] {$u_2$};
				\draw[fill=black!30!green, fill opacity=0.1,shift={(-0.17 cm,-3.75 cm)},rotate=60] (5,0) ellipse (2 and 1);
				\draw[thick, red, ->](0,0) -- (2.325,0.59);
				\node[red,label={below:$\red{R}$}] (R) at (1.325,0.53){};
			\end{tikzpicture} &
			\caption{Two ellipses with their centers $x_1, x_2 $, the distance between the centers $R$ and their principal axes $u_1, u_2$ respectively.}   \label{fig:rep_ellipsoid}
		\end{tabular} }
\end{figure}
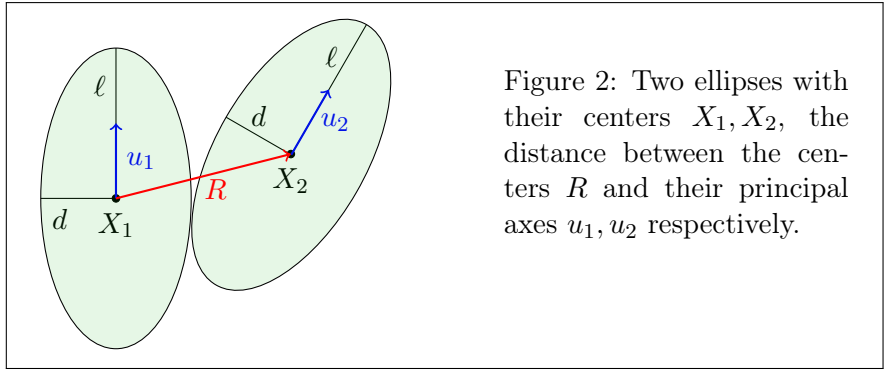 
In particular, notice that the potential $V_b$ acts also on the distance between the centers of the particles. This effect does not appear in the classical nematic-alignment model \eqref{eq:nematic_X}-\eqref{eq:nematic_pot}. Our goal is to investigate how this potential affects the mean-nematic direction and the positions of the particles. To carry this out, we derive a kinetic equation (see Section \ref{ssec:kinetic_equation}) for system \eqref{eq:discrete_system_X}-\eqref{eq:discrete_system_u} as a first step and subsequently we obtain continuum equations from the kinetic equation (see Section \ref{sec:macro_limit}).

Let us have a closer look at the interaction potential. We describe the binary interactions between identical particles (ellipses for $n=2$, or spheroids for $n=3$, see Figures \ref{fig:spheroids} and \ref{fig:rep_ellipsoid}) via an anisotropic Gaussian-type potential $V_b$ of the form
\begin{align} \label{eq:potentialG}
	V_b (u_1, u_2, R) = (4 \pi)^{-n/2} b(u_1,u_2)\exp \left ( - R^T \Sigma^{-1}R \right ),
\end{align} where $b(u_1,u_2)\geq 0$ is a scaling factor, $R=x_2-x_1$ is the vector between the centers of two particles, and the over index `T' denotes the vector transpose. The matrix $\Sigma$ is given by
$$\Sigma = \gamma_1 +\gamma_2,$$
where  $$\gamma_i = (\ell^{2}-d^2) u_i\otimes u_i +d^2 \Id, \qquad i \in \{1,2\}.$$
Our particular choice of the scaling factor $b$ is given by
\begin{align} \label{eq:scaling_factor_WG}
	b_{\text{WG}}(u_1,u_2):= \lp 1-\chi^2(u_1\cdot u_2)^2\rp^{1/2}=\mbox{det}\lp \Sigma\rp ^{1/2}.
\end{align}
Using the scaling factor \eqref{eq:scaling_factor_WG}, the weighted Gaussian potential \eqref{eq:potentialG} takes the following form:
\begin{align} \label{eq:weighted_Gaussian_potential}
	V_{b_{\text{WG}}} (u_1, u_2, R) :=  (4\pi)^{-n/2}\mbox{det}(\Sigma) ^{1/2} \exp\lp - R^T \Sigma^{-1}R \rp.
\end{align}

We dedicate Section \ref{ssec:potentials} to a detailed explanation of repulsive potentials and some numerical simulations to show their effect on the dynamics of interacting particles. This section also includes a more detailed motivation for our particular choice of the scaling factor \eqref{eq:scaling_factor_WG}. Before we explain the kinetic equation arising from the particle dynamics, we make the following remark:

\begin{remark} \label{rem:Berne_Pechukas}
	In \cite{BP72}, Berne and Pechukas considered the potential $U$ for a single spheroid with its center $x$ and direction of its main axis $u$:
		\begin{align*}
		\bar U(x):=\exp(-x^T\gamma^{-1}x),
	\end{align*}
	where $\gamma= (\ell^2-d^2)u\otimes u +d^2\Id$. Using the potential $\bar U=\bar U(x)$, they introduced the anisotropic Gaussian-type potential \eqref{eq:potentialG} to describe the binary repulsive interactions. The level sets where $\bar U$ is constant correspond to ellipsoids of revolution about the axis $u$,  i.e., spheroids concentric to the original  $(x,u)$-spheroid, and they remain so even if the potential $\bar U$ is multiplied by a factor that could potentially depend on $u$. For this reason, introducing the scaling factor $b=b(u_1,u_2)$ does not change the intrinsic properties of the potential. On the other hand, this offers some simplifications for mathematical analysis, see Section \ref{ssec:potentials} for more details. 
\end{remark}

\subsection{Kinetic equation and macroscopic quantities} \label{ssec:kinetic_equation}

In this section, we write the mean-field equation (or large-particle limit) of the particle system \eqref{eq:discrete_system_X}-\eqref{eq:discrete_system_u}. We consider a system of $N\geq 0$ identical particles in $\R^n$ for $n \in \{2,3\}$, identified by the positions of their centers and the directions of their main axes $(X_i, U_i)_{i=1,\dots,N}$ that follow the dynamics given by the system \eqref{eq:discrete_system_X} -\eqref{eq:discrete_system_u}. As $N\to \infty$, the particle system is described by the probability distribution function $f=f(t,x,u)$ of particles at position $x$ with the main axis in direction $u$ at time $t$. Obtaining the equation for $f$, at least formally, uses standard tools from kinetic theory, which we will not detail here but refer the reader to, e.g., \cite{V01}. 

We start with defining an empirical measure $f^N$ of the particles given by
\begin{align} \label{eq:def:empirical_distribution}
	f^N(t,x,u)=\frac{1}{N}\sum_{i=1}^N \delta_{(X_i(t),U_i(t))} (x,u),   	
\end{align}
where $\delta_{(X_i(t),U_i(t))} (x,u)$ is the Dirac delta distribution at $(X_i(t),U_i(t))$. The formal mean-field limit of the kinetic equation for the empirical distribution under the assumption that $f^N \rightarrow f$ as $N \rightarrow \infty$ solves
\begin{align} \label{eq:kinetic_eq}
	\p_t f-\mu \nabla_x\cdot \left ( \left  (\nabla_x V_f \right ) f \right )-\lambda\nabla_{u}\cdot \left ( \left (\nabla_u V_f \right ) f \right )-D_x\Delta_x f-D_u\Delta_{u} f=0,
\end{align} where
\begin{align} \label{eq:potential_V_f}
	V_f(t, x_1,u_1) :=\ints \intx V_b(u_1, u_2, x_2-x_1) f (t, x_2, u_2) \d x_2 \d u_2.
\end{align}
Next, we define the macroscopic quantities associated to $f=f(t,x,u)$, namely the spatial (or mass) density of the particles $\rho_f(t,x)$ and the mean-nematic direction $\O_f(t,x)$ at time $t \geq 0$ at position $x\in \R^n$. Our goal is to derive equations for $\rho_f(t,x)$ and $\O_f(t,x)$ from $f=f(t,x,u)$. 

The mass density $\rho_f(t,x)$ is defined by 
\begin{align*}
	\rho_f(t,x):= \ints f(t,x,u) \d u. 
\end{align*} 
To define $\O_f(t,x)$, we need to consider the following:
\begin{align} \label{eq:Qdef}
	(\rho_f  Q_f)(t,x): = \int_{\mathbb S^{n-1}} \left ( u\otimes u-\frac{1}{n}\Id  \right ) f(t, x,u) \d u.
\end{align} where $Q_f(t,x)$ is a matrix associated to the distribution $f=f(t,x,u)$ and it is called \emph{$Q$-tensor} in the literature of liquid crystals see, e.g., \cite{B17,V19}. Note that $Q_f$ is a symmetric, trace-free matrix satisfying $Q_f \geq -\frac{1}{n} \Id$. 

The principal eigenvector of $Q_f(t,x)$ gives the direction that corresponds to the mean-nematic direction of the particles located at $x\in \mathbb{R}^n$ at time $t \geq 0$ (see, e.g., page 15 of \cite{DFM-AT18} for an explanation). We denote the principal eigenvector by $\Omega_f(t,x) \in \mathbb{S}^{n-1}$ as long as it is unique (up to a change of sign). We have, in particular, that $\Omega_f=\Omega_f(t,x)$ maximizes over $\mathbb{S}^{n-1}$ the quantity (see page 37 of \cite{DFM-AT18}),
\begin{align*}
u  \mapsto  u\cdot (\rho_f Q_f) u = \ints \lp ( u \cdot \bar u)^2-\frac{1}{n}\rp f(t,x,\bar u) \d \bar u.
\end{align*}
This can also be seen as a consequence of the Courant-Fisher theorem, or the min-max theorem, see, e.g., \cite{T20}. Indeed, notice that if $f=f^N$ is the empirical measure \eqref{eq:def:empirical_distribution}, the previous expression implies that $\Omega_f$ maximizes
\begin{align*}
	u\mapsto\frac{1}{N}\sum_{j=1}^N (u\cdot U_j)^2,
\end{align*}
which is analogous to maximizing
\begin{align*}
	\frac{1}{N}\sum_{j=1}^N V_{\text{nem}}(u,U_j),
\end{align*}
where $V_{\text{nem}}$ is the potential given in Equation \eqref{eq:nematic_pot}. Thus, $\Omega_f$ defined in this way is indeed the mean-nematic direction. 

\bigskip

In this article, we derive equations for the spatial density $\rho=\rho_f(t,x)$ of particles and their mean-nematic direction $\Omega=\Omega_f(t,x)$. Before delving into the details in the next section, we end this introductory part with some comments on the continuum equations. For the density $\rho$, we obtain an equation of the form
\begin{align*}
	\p_t \rho = D_x\Delta_x \rho +\mu \nabla_x \cdot \left (K(\rho)\rho\nabla_x \rho\right ),
\end{align*} where $K$ is a functional that depends on $\rho$. Notice that the particle density does not depend on the mean-nematic direction $\Omega$. However, the effect of the repulsive potential is present through the functional $K$, which depends, particularly, on the anisotropy parameter $\chi^2$ (defined in \eqref{eq:def-chi}).

The equation for $\Omega$ is a combination of transport and diffusion, with a cross-diffusion term in $\rho$. Intriguingly, the well-posedness of the equation imposes a constraint on the parameters. Particularly, we require that 
\begin{align}
	\frac{D_x}{\mu}>\frac{D_u}{\lambda}.
    \label{eq:intro_condition_wellposedness}
\end{align} 
However, this constraint is not required at the particle level. We discuss in detail possible explanations for this constraint in Section \ref{ssec:interpretation}. Our main result is given in Theorem \ref{th:macroscopic limit} and its interpretation is given in Section \ref{ssec:interpretation}.

\paragraph{Structure of the paper.} In Section \ref{sec:macro_limit}, we state our main result, Theorem \ref{th:macroscopic limit}, after some preliminary definitions and results. Theorem \ref{th:macroscopic limit} is followed by an extensive interpretation section, Section \ref{ssec:interpretation}, where we comment on each component of our continuum equations. In Section \ref{ssec:potentials}, we motivate our choice of the Gaussian-type repulsive potential \eqref{eq:weighted_Gaussian_potential} by comparing it with similar interaction potentials used in the literature. Section \ref{ssec:particle_sims} is dedicated to numerical simulations of the stochastic particle system \eqref{eq:discrete_system_X}-\eqref{eq:discrete_system_u}. Here, we conduct a detailed parameter study and a numerical comparison of the effect of different repulsive potentials on the global alignment of the particles. We dedicate Section \ref{sec:proof} to the proof of Theorem \ref{th:macroscopic limit}. Finally, the paper is complemented with some concluding remarks and perspectives in Section \ref{sec:conclusions}, followed by appendices. 

\section{Continuum equations}
\label{sec:macro_limit}

In this section, we derive continuum equations, namely the equations for the mass density $\rho_f(t,x)$ and the mean-nematic direction $\O_f(t,x)$ from the kinetic equation \eqref{eq:kinetic_eq}. We start with non-dimensionalizing the kinetic equation \eqref{eq:kinetic_eq}. To this end, we introduce a scaling parameter $0 < \eps \ll 1$ and denote by $f^\eps$ the solution of the scaled kinetic equation.

\paragraph{Non-dimensionalization.} Introducing the units of space $x_0$ and time $t_0$, we define the dimensionless variables,
\begin{align*}
	\tilde{x}=\frac{x}{x_0}, \quad \tilde{t}=\frac{t}{t_0}, \quad \tilde{f}=f x_0^n, 
\end{align*} and the dimensionless parameters,
\begin{align*}
\tilde{\mu}=\mu \frac{t_0}{x_0^2}, \quad \tilde{\lambda}=t_0 \lambda, \quad \tilde{D}_u=t_0 D_u, \quad \tilde{D}_x=\frac{D_x t_0}{x_0^2}, \quad 
\tilde{\ell}=\frac{\ell}{x_0}, \quad \tilde{d}=\frac{d}{x_0}.
\end{align*}We remark that the anisotropy factor $\chi$ and the potential $V_b$ are already dimensionless. Thus, we have 
\begin{align*}
	\tilde{\chi}=\chi, \quad \tilde{V_b}=V_b, \quad \tilde{V_f}=V_f.
\end{align*}
Using these dimensionless quantities, dropping the tildes for the sake of simplicity, we obtain exactly the same equation as before, i.e., we end up with Equation \eqref{eq:kinetic_eq}. But now all the variables and parameters are without units.

\paragraph{Scaling.} Considering a scaling parameter $0 < \eps \ll 1$, we introduce the primed variables below,
\begin{align} \label{eq:scaling}
	\ell = \eps \ell', \quad d = \eps d', \quad \mu= \eps^{-n} \mu', \quad \lambda = \eps^{-(n+a)}\lambda', \quad D_u = \eps^{-a} D_u',
\end{align} where $a \in (0,2]$ is a constant.

Notice that the scaling in $\ell$ and $d$ combined with the scaling of the potential by  $\eps^n V^\eps_f =V_f$ produces a localization in space of the potential $V_f$ (for details, see Lemma \ref{lem:rescaling the potential}). The choice of $a$ has an influence on the resulting macroscopic equation (see Remark \ref{rem:choice_of_a}). Notice also that the diffusion constant $D_x$ stays of order $1$. We obtain the following rescaled kinetic equation (skipping the primes for simplicity):
\begin{align} \label{eq:dimless_kinetic_eq_without_expansion}
	\p_t f^\eps-\mu \nabla_x\cdot ((\nabla_x V^\eps_{f^\eps})f^\eps)-D_x\Delta_x f^\eps =\frac{1}{\eps^a} \left( \lambda\nabla_{u}\cdot ((\nabla_u V^\eps_{f^\eps}) f^\eps)+D_u\Delta_{u} f^\eps\right) . 
\end{align}
Next, we further expand the potential $V_f$.
\begin{lem}[Expansion of the potential] \label{lem:rescaling the potential}
	Considering $\ell = \eps \ell'$ and $d = \eps d'$ we have the following expansion $V_f^\eps$ of the scaled weighted Gaussian potential $V_f$ (defined by using  \eqref{eq:weighted_Gaussian_potential}):
	\begin{align} 
		V^\eps_f (t,x_1, u_1):=\frac{1}{\eps^n} V_f(t, x_1,u_1) \nonumber
		=&\ints\lp 1-\chi^2 (u_1\cdot u_2)^2\rp f(t, x_1,u_2) \d u_2 \\
		&+\frac{ \eps^{2}}{4} \ints\lp 1-\chi^2 (u_1\cdot u_2)^2\rp \Sigma (u_1,u_2): \nabla^2_x f(t, x_1, u_2) \d u_2 +\mc O (\eps^{3}) \nonumber \\
		=& \, W_f(t,x_1,u_1) + \eps^2 B_f(t,x_1,u_1) + \mathcal{O}(\eps^3), \label{eq:rescaleV}
 	\end{align} where $\nabla^2_x$ is the Hessian, i.e., a $n\! \times\! n$-matrix with components $(\nabla^2_x)_{ij}=\partial_{x_i}\partial_{x_j}$, 
  `$:$' denotes the double contraction, i.e., $A : B := \sum_{i,j} A_{ij} B_{ij}$, $W_f$ and $B_f$ are defined as 
\begin{align} 
	W_f (t,x_1, u_1) &:= \ints \left ( 1- \chi^2 (u_1 \cdot u_2)^2\right )f(t,x_1, u_2) \d u_2, \label{eq:W_f} \\
	B_f (t,x_1, u_1)  &:= \frac{1}{4} \ints\left ( 1- \chi^2 (u_1 \cdot u_2)^2\right ) \Sigma (u_1,u_2): \nabla^2_x f(t,x_1, u_2) \d u_2. \label{eq:B_f}
\end{align}
\end{lem}
\begin{proof}
The proof can be found in Appendix \ref{appendix_lem rescaling the potential}.
\end{proof}

 Nematic alignment potentials analogous to $W_f$ appear in the literature of collective dynamics, e.g., in a model for nematic alignment of fibers \cite{DDP16}, in a model for body attitude coordination \cite{DFM-AT18}, and in a model of pure nematic alignment \cite{DM-A20}.

\medskip 

Finally, using the expansion for the scaled weighted Gaussian potential $V_{f}^\eps$ in Lemma \ref{lem:rescaling the potential} we obtain:
\begin{multline} \label{eq:dimless_kinetic_eq}
	\p_t f^\eps-\mu \nabla_x\cdot ((\nabla_x W_{f^\eps})f^\eps)-D_x\Delta_x f^\eps- \eps^{2-a}\lambda \nabla_{u}\cdot ((\nabla_u B_{f^\eps}) f^\eps) =\frac{D_u}{\eps^a} C(f^\eps) +\mathcal{O}(\eps^{\min\{2,3-a\}}), 
\end{multline} with $a\in (0,2]$ and where 
\begin{align} \label{eq:collision_operator}
	C(f):=\frac{\lambda}{D_u}\nabla_u\cdot((\nabla_u W_{f}) f)+\Delta_{u} f .
\end{align}

\begin{remark}[Motivation for the scaling factor $a$ and scaling choices] \label{rem:choice_of_a} 
The scaled kinetic equation with $a = 2$ can also be equivalently obtained by just scaling space and time by $x' = \sqrt{\eps} x$ and $t' = \eps t$, which corresponds to a diffusive scaling. However, we also consider an alternative scaling (corresponding to $a\in (0,2)$) where the term $B_f$ vanishes in the macroscopic limit.  The term $B_f$ is specific to the Gaussian repulsive potential, and it acts as a correction to the primary nematic alignment potential $W_f$. As we will see, the term $B_f$ is unique in the sense that it gives different behavior for prolate and oblate particles (see term $\Pi_3$ in the macroscopic limit \eqref{def:Pi3}). 
The two scaling factors that we consider highlight the differences in the dynamics with and without the term $B_f$. 

The scaling for $\ell$ and $d$ presented in \eqref{eq:scaling} makes the ellipsoid's size vanish as $\varepsilon\to 0$, while keeping the aspect ratio. This scaling choice can be interpreted as doing a `zoom out' in space. The scaling for $\lambda$, $D_u$ makes the term acting on the orientation $u$ predominant, while  keeping the ratio between the potential force and the noise intensity $\lambda V_f/D_u$  invariant under rescaling. Finally the parameter $\mu$ is rescaled so that the total potential force acting on the positions of the particles $\mu V_f$ is invariant under rescaling. 
\end{remark}

Our goal is to derive equations for the time evolution of the mass density $\rho_f= \rho_f(t,x)$ and the mean-nematic direction $\Omega_f=\Omega_f(t,x)$. For this reason, we  rewrite the kinetic equation in terms of the $Q$-tensor $Q_f$. First, using that $(u\cdot v)^2=u^T (v\otimes v) u$ and that $\rho_f = \ints f \d u$, we rewrite $W_f$ in \eqref{eq:W_f} as follows:
\begin{align} \label{def:W_f}
	W_f(t, x_1, u_1) = -\chi^2 u_1^T (\rho_f Q_f) u_1 - \left( \frac{\chi^2}{n}-1\right) \rho_f .
\end{align}
With this new expression for $W_f$ and using that $\nabla_u \rho_f=0$, the kinetic equation \eqref{eq:dimless_kinetic_eq} becomes
\begin{multline}\label{eq:rescaled_kinetic}
	\partial_t f^\eps + \mu \chi^2 \nabla_x \cdot \left (\nabla_x \right (u^T (\rho_{f^\eps} Q_{f^\eps} ) u \left  ) f^\eps \right ) + \mu\lp\frac{\chi^2}{n}-1\rp \nabla_x \cdot ((\nabla_x \rho_{f^\eps}) f^\eps) \\ - \lambda \eps^{2-a}\nabla_u \cdot ((\nabla_u B_f) f^\eps) - D_x \Delta_x f^\eps = \frac{D_u }{\eps^a}C(f^\eps)+ \mathcal{O}(\eps^{\min\{2,3-a\}}),
\end{multline} where
\begin{align}
	\label{eq:defC}
	C(f) =\nabla_u \cdot \left(\nabla_u f - \frac{\lambda}{D_u}f\nabla_u \left (\chi^2u^T (\rho_f Q_f )u \right ) \right).
\end{align}

In the next section, we give some preliminary definitions and concepts to analyze the operator $C$ in detail.

\subsection{Properties of the operator \texorpdfstring{$C$}{C}} \label{ssec:operator_C}

From Equation \eqref{eq:rescaled_kinetic}, we can observe that, at least formally, if $f^\eps$ converges to some function $f^0$ as $\eps \to 0$, then it must hold that $f^0$ belongs to the kernel of the operator $C$, i.e., $C(f_0)=0$. Fortunately, the equilibria for $C$ have already been studied in \cite{DFL22} and in this section we summarize the results presented there. Therefore, this section mainly follows some results of \cite{DFL22} omitting their proofs.

In particular, we are interested in stable equilibria of the operator $C$ (for an exact notion of stability see \cite{DFL22}). The key result, containing the characterization of the stable equilibria of $C$, is given in Lemma \ref{lem:stable_equilibria}, which can be found at the end of this section. 

\begin{remark} \label{rem:C}
	The operator $C$ defined in \eqref{eq:defC} is equivalent to the one presented in Equation (39) in \cite{DFL22} for
	\begin{equation} \label{eq:def:alpha}
		\alpha:=\frac{\chi^2\lambda}{D_u}.
	\end{equation}
\end{remark}
Next, we introduce some preliminaries. 
\begin{dfn}[Uniaxial tensor]\label{def:uniaxial}
		Given $\O \in \mathbb S^{n-1}$,  the normalized, uniaxial, trace-free tensor $A_\O$ in the direction of $\O$ is defined by
	\begin{align} \label{def:A_Om}
		A_\O = \O \otimes \O - \frac{1}{n} \Id. 
	\end{align} 
\end{dfn}
The tensor $A_\O$ is symmetric with principal eigenvalue $\frac{n-1}{n}$. It is called \emph{uniaxial} since it has only two eigenvalues and one of them is simple. The normalized eigenvectors associated to the simple eigenvalue are $\pm \O$. 
 
Next, we define the Gibbs distributions of uniaxial tensors. 
\begin{dfn}[Gibbs distribution of an uniaxial tensor] \label{def:Gibbs distribution}
Given $\eta>0$ and the tensor $A_\O$, the Gibbs distribution $G_{\eta A_\O}$ associated to $\eta A_\O$ is defined as 
\begin{align} \label{def:GetaA}
	G_{\eta A_\O} (u) = \frac{1}{Z_\eta} e^{\eta (u \cdot \O)^2}, \quad \text{where} \quad  Z_\eta= \int_{\mathbb S^{n-1}} e^{\eta (u \cdot \O)^2} \d u. 
\end{align}
\end{dfn} Finally, the order parameter associated to a Gibbs distribution is defined as follows. 
\begin{dfn}[Order parameter] \label{def:order_parameter}
The order parameter associated to $G_{\eta A_\Omega}$ is denoted by $S_2=S_2(\eta)$ and given by
	\begin{align} \label{eq:def:S_2}
		S_2 (\eta) :=\frac{n}{n-1}\Omega^T Q_{G_{\eta A_\O}} \Omega = \frac{n}{n-1}\ints\lp (u\cdot\O)^2-\frac{1}{n} \rp G_{\eta A_
			\Omega} \d u,
	\end{align}
 where $Q_{G_{\eta A_\O}}$ is the $Q$-tensor associated to $G_{\eta A_\O}$.
\end{dfn} 
\begin{remark} Using the change of variables \eqref{eq:change_of_variables}, the order parameter $S_2 (\eta)$ can be written as 
	\begin{align}\label{eq:def:S_2_theta}
	S_2(\eta) = \frac{n}{n-1}\int_0^\pi\lp \cos^2\theta-\frac{1}{n} \rp g_{\eta}(\theta) \sin^{n-2}\theta \d \theta.
	\end{align} where the function $g_\eta$ is given in \eqref{lem:var_trafo_G_h}.
\end{remark} Notice that $S_2(\eta)$ is actually independent of $\Omega$ and it satisfies the following:
 \begin{lem}[Proposition 2 in \cite{DFL22}]
	\label{lem:properties_S2}
	It holds that
	\begin{align*}
		Q_{G_{\eta A_\Omega}}=S_2(\eta) A_\Omega.
	\end{align*}
	Moreover, $S_2:(0,\infty) \to (0,1)$ is  non-decreasing with $S_2\to 0$ as $\eta \to 0$ and $S_2 \to 1$ as $\eta \to \infty$.
\end{lem}
Next, we define the function $\eta=\eta(\rho)$ implicitly through the following equation:
\begin{prp}[Implicit definition of $\eta=\eta(\rho)$, Proposition 3 in \cite{DFL22}]	\label{prop:def_eta}
	The equation 
	\begin{align}  \label{eq:etaroots}
		\frac{\eta}{\alpha \rho} = S_2(\eta)
	\end{align} has at least a root $\eta$ if and only if $\rho \in (\rho_*, +\infty)$, $\rho_*>0$. Moreover, \eqref{eq:etaroots} has at most two roots. Choosing the largest root, we can define a smooth, non-decreasing function  $(\rho_*, +\infty) \to (\eta^*, +\infty), \linebreak \rho \mapsto \eta(\rho)$, where $\eta^* = \lim_{\rho \to \rho_*} \eta(\rho) \geq 0$.
\end{prp}

Finally, with all the definitions above, we can state the main result regarding the equilibria of $C$:
\begin{lem}[Stable equilibria of $C$, Lemma 4.4 in \cite{DFL22}] \label{lem:stable_equilibria}
		Let $n \in \{2,3\}$. For $\rho_*$ given in Proposition \ref{prop:def_eta} the following holds:
		\begin{itemize}
			\item[(i).] If $\rho<\rho_*$, then $f=\rho$ (the uniform equilibrium) is the only stable equilibrium of the operator $C$.
			\item[(ii).] If $\rho>\rho_*$, then the only stable equilibria of the operator $C$ are of the form 
			\begin{align*}
				f=\rho G_{\eta(\rho)A_\Omega}.
			\end{align*} In particular, we have 
		\begin{align} \label{eq:consistency_relation}
			Q_f = S_2(\eta(\rho))A_\Omega, 
		\end{align} where $S_2$ and $\eta=\eta(\rho)$ are defined  in \eqref{eq:def:S_2} and \eqref{eq:etaroots} respectively. 
	Equation \eqref{eq:consistency_relation} is referred to as \emph{“consistency relation”} and can be written, using \eqref{eq:etaroots}, equivalently as 	
	\begin{align} \label{eq:limit_rhoQ}
		\alpha\rho Q_{\GetaA} = \eta(\rho) A_\Omega.
	\end{align}
		\end{itemize}
\end{lem}

\begin{remark} \label{rem:the_nematic_potential}
	The operator $C$ is analogous to the operator that appears for models of nematic alignment where the nematic alignment is imposed, see, e.g., \cite{DM-A20}. However, in \cite{DM-A20} the potential is written in a so-called, `mean-field force' form rather than a binary force. Here, we have the term
	\begin{align*}
		\varphi_1(u):=u^T \rho_f Q_f u = \ints \lp (u\cdot \bar u)^2 - \frac{1}{n}\rp f(t,x,\bar u) \d \bar u,
	\end{align*}
	while in \cite{DM-A20} this operator is replaced by
	\begin{align*}
		\varphi_2(u):=u^T A_{\Omega_f}u=(u\cdot \Omega_f)^2-\frac{1}{n},
	\end{align*}
	where $\Omega_f$ corresponds to the mean-nematic direction (that is why it is called `mean-field force'). We expect that both cases lead to nematic alignment. Already in the case $\varphi_2$ we see that $u=\pm \Omega_f$ is the maximizer. In the first case, we also know that $\Omega_f$ corresponds to the principal eigenvector of $Q_f$, gives the mean-nematic direction and maximizes $\varphi_1$. However, it is not guaranteed that the principal eigenvalue has only multiplicity $1$. This adds an extra degree of difficulty to the analysis. Lemma \ref{lem:stable_equilibria} demonstrates that this leads to phase transitions in the macroscopic limit, with spatial regions of disordered dynamics, i.e., the dynamics where a well-defined principal eigenvector does not exist, and thus, particles do not align. This takes place in the regions of low density, i.e., $\rho < \rho_*$.
\end{remark}

\begin{remark}
	Given the matrix $Q_f$, suppose that $\Omega_f\in \mathbb{S}^{n-1}$ is the principal eigenvector (assumed to be unique up to a change in sign). Then, it holds that $A_{\Omega_f}$ also has a principal eigenvector $\Omega_f$. So, both $Q_f$ and $A_{\Omega_f}$ give the same mean-nematic direction. However, $Q_f \neq A_{\Omega_f}$ in general. In particular, the principal eigenvalue of $A_{\Omega_f}$ is always $\frac{n-1}{n}$, while for $Q_f$ the eigenvalues have $\frac{n-1}{n}$ as an upper bound (see \cite{DM-A20}). In particular (see \cite{DFL22}), to measure the degree of nematic alignment one considers the following order parameter $\gamma_f$:
	\begin{align} \label{def:deg_align}
		\gamma_f = \frac{n}{n-1} \beta_f  \in (0,1],
	\end{align} where $\beta_f$ is the largest eigenvalue of $Q_f$. If $f$ is uniformly distributed on the sphere (particles are fully misaligned), then $\gamma_f$ is close to $0$. On the contrary, if $f$ is close to $\frac{1}{2}(\delta_\Omega +\delta_{-\Omega})$, i.e., particles are fully aligned and $Q_f = A_{\Omega_f}$, then  $\gamma_f=1$. So the case $Q_f=A_{\O_f}$ is an extreme case that corresponds to having all particles nematically aligned in the direction $\Omega_f$. 

In fact, the function $S_2=S_2(\eta)$ is the order parameter of $ G_{\eta A_\Omega}$.
\end{remark}

\begin{remark} As $\eta \to 0$, $G_{\eta A_\O}$ converges to the uniform probability distribution on $\mathbb S ^{n-1}$ and $S_2$ converges to $0$. As $\eta \to \infty$,  $G_{\eta A_\O}$ converges to two Dirac delta distributions $\frac{1}{2}(\delta_\O +\delta_{-\O})$ which accounts for fully aligned distribution in the direction of $\O$ and thus also $S_2$ converges to $1$. Therefore, as $\eta$ increases, $S_2$ increases too (see also Figure \ref{fig:eta_and_S_2_b}), and $G_{\eta A_\O}$ shows increasing order of alignment.
\end{remark}

Now, we are ready to state our main result in the next section. 

\subsection{The macroscopic limit} \label{ssec:macro_lim}

\begin{asm} \label{as:A}
We assume throughout that all the functions are as smooth as needed so that all the limits exist and the convergences are as strong as needed.
\end{asm}

 Next, we introduce the following definition which we will use in the theorem below.
\begin{dfn} \label{dfn:tensor_operations}  Let  $A$ be a $5$-tensor and $B$ a $4$-tensor in $\R^n$, then we define the following operation between tensors:
\begin{align*}
([A:B]_{[2,3,4,5:1,2,3,4]})_p:=\sum_{j,k,l,m=1}^n A_{pjklm}B_{jklm}, \quad p=1,\hdots,n, 
\end{align*} The contraction with the subscript $[2,3:1,2]$ is analogously defined.
\end{dfn}

We give our main result in the following theorem:

\begin{thm} \label{th:macroscopic limit}
	Let $n \in\{2,3\}$, $a\in(0,2]$, and $f^\eps=f^\eps(t,x,u)$ be the solution to the kinetic equation \eqref{eq:rescaled_kinetic}. Suppose that Assumption \ref{as:A} holds and that $f^\eps (t,x,u)\to f^0(t,x,u)$ as $\eps \to 0$. Consider $(t,x)$ such that $\int_{\S^{n-1}} f^0 du >\rho_*$ (where $\rho_*$ is given in Proposition \ref{prop:def_eta}). Then, in this domain, it holds that $f^0=\rho(t,x)\GetaA$ with $\rho(t,x)=\int_{\S^{n-1}} f^0 du >\rho_*$,  where $A_\O, \GetaA, \eta(\rho)$ are given in \eqref{def:A_Om}, \eqref{def:GetaA} and \eqref{eq:etaroots} respectively. Moreover, the mass density $\rho(t,x)$ and the mean-nematic direction of the particles $\O(t,x)$ satisfy the following system of equations
\begin{subequations}
\begin{align}
		&\p_t \rho = D_x\Delta_x \rho +\mu \nabla_x \cdot \left(K(\eta(\rho))\rho\nabla_x \rho\right), \label{eq:macro_rho}\\
		 &\p_t \O   + \mu \Pi_2(\rho)(\nabla_x \rho\cdot\nabla_x) \O  + \mu \left ( \sigma-\nu \right ) P_{\O^\perp} \Delta_x \O   = \1_{a=2}\lambda \Pi_3 (\O, \rho). \label{eq:macro_Omega}
\end{align}
\end{subequations}where 
\begin{align} \label{eq:Ceta_coeff_porousmedium}
	K(\eta(\rho)):=  1- \frac{\chi^2}{n}-\sigma\frac{n-1}{n}S_2(\eta(\rho))\eta'
(\rho), 
\end{align} and
$$\sigma:= \frac{D_u}{\lambda}, \qquad \nu := \frac{D_x}{\mu}.$$
Moreover, $\Pi_2$ is given by
\begin{align}
	\Pi_2(\rho)= \frac{\sigma-2\nu}{\rho} + \frac{\chi^2}{n}-1  + 2 \frac{\eta'(\rho)}{\eta(\rho) } \lp \sigma-\nu \rp +   \eta'(\rho)\left ( 2 \lp \sigma-\nu \rp\frac{c_{3,2}(\rho)}{c_{1,2}(\rho)}-d_{2,0}(\rho)\left ( \sigma-2\nu \right )-\frac{\sigma}{n} \right ),
\end{align} where, for $k,p\in \mathbb{N}\cup \{0\}$,  $c_{k, p}=c_{k,p}(\rho)$ and $d_{k,p}=d_{k,p}(\rho)$ are given by 
\begin{align}
	c_{k,p}(\rho) &:=\int_0^\pi \cos^k\theta g_{\eta(\rho)}(\theta) h_{\eta(\rho)}(\cos\theta) \sin^{n-2+p}\theta \d \theta, \label{eq:def_ckl}\\
	 d_{k,p}(\rho) &:= \int_0^\pi \cos^k\theta g_{\eta(\rho)}(\theta) \sin^{n-2+p}\theta \d \theta, \label{eq:def_dkl}
\end{align} 
where the function $h_\eta$ is defined in Proposition \ref{prop:GCI_property} and the function $g_\eta$ is given in \eqref{lem:var_trafo_G_h}.
It holds that $\frac{c_{3,2}}{c_{1,2}} \geq 0$ and  $d_{2,0}\geq 0$. 

In the case of $a=2$, the right hand side of \eqref{eq:macro_Omega} does not vanish and $\Pi_3$ is given by 
\begin{multline} \label{def:Pi3}
	 \frac{8c_{1,2}(\rho)\eta(\rho) }{(n-1)}\Pi_3(\Omega,\rho) =  (\ell^2 - d^2) [(H^r_2 : (D^2_x \rho) )]_{[2,3:1,2]} -  2 \chi^2 d^2 [H^r_2 : \Delta_x (\rho H_2)]_{[2,3:1,2]} \\
	  - \chi^2 (\ell^2 - d^2) \left ( [H^r_4: (D^2_x \otimes (\rho H_2))]_{[2,3,4,5:1,2,3,4]}  + [H^r_2\otimes D^2_x : (\rho H_4)]_{[2,3,4,5:1,2,3,4]} \right )
\end{multline} where in the case of dimension  $n=2$ we have that
\begin{align}
	H_2 &= d_{2,0} (\Omega\otimes \Omega) + d_{0,2}(\Omega^\perp \otimes \Omega^\perp), \label{eq:H2d2}\\
	H_2^r &= d_{2,2} \Omega^\perp \otimes\left( (\Omega^\perp \otimes  \Omega) + (\Omega  \otimes \Omega^\perp)\right), \label{eq:H2rd2}\\
	H_4 &= d_{4,0} (\Omega\otimes \Omega \otimes \Omega\otimes \Omega) + d_{2,2} S_{2\Omega,2\Omega^\perp}(\Omega,\Omega^\perp)+ d_{0,4} (\Omega^\perp \otimes \Omega^\perp\otimes \Omega^\perp\otimes \Omega^\perp), \label{eq:H4d2}\\
	H_4^r &=\Omega^\perp \otimes \left( d_{4,2} S_{3\Omega,\Omega^\perp}(\Omega,\Omega^\perp) + d_{2,4} S_{\Omega,3\Omega^\perp}(\Omega,\Omega^\perp)\right), \label{eq:H4rd2}
\end{align} with 
\begin{align} \label{eq:S22}
S_{2\O,2\O^\perp}(\O,\O^\perp) &:= (\O \otimes \O^\perp + \O^\perp \otimes  \O)  \otimes (\O \otimes \O^\perp + \O^\perp \otimes  \O) \\
& \quad +  \O \otimes \O \otimes \O^\perp \otimes \O^\perp + \O^\perp \otimes \O^\perp \otimes \O \otimes \O,
\end{align} and
\begin{align} \label{eq:S31}
		S_{3\O,\O^\perp}(\O,\O^\perp) := (\O \otimes \O) \otimes (\O \otimes \O^\perp + \O^\perp \otimes  \O) + (\O \otimes \O^\perp + \O^\perp \otimes \O) \otimes (\O \otimes \O) 
\end{align} and $S_{\Omega,3\Omega^\perp}$ is defined analogously as $S_{3\Omega,\Omega^\perp}$ (by exchanging the values of $\Omega$ and $\Omega^\perp$).

In dimension $n\geq 3$, the values of $H_2$, $H_2^r$, $H_4$, $H_4^r$ are more complex and are given in Proposition \ref{lem:Hdimension3} below. 
\end{thm}

\begin{prp}[Dimension $n\geq 3$] 
\label{lem:Hdimension3}

For $n\geq 3$ the functions $H_2$, $H_2^r$, $H_4$, $H_4^r$ are given by 
\begin{align}
H_2(\eta,\Omega,\Omega^\perp) &= d_{2,0}(\Omega\otimes \Omega) + \frac{d_{0,2}}{n-1}\pp \label{eq:H2inlemma}\\
H_2^r(\eta, \Omega,\Omega^\perp) &= \frac{d_{2,2}}{n-1}\left( \pp\otimes \Omega + [\pp \otimes \O\otimes \pp]_{:24} \right) \label{eq:H2rinlemma}\\
H_4(\eta,\Omega,\Omega^\perp) & = d_{4,0} (\Omega\otimes \Omega\otimes \Omega\otimes \Omega)+ \frac{d_{0,4}}{(n-1)(n+1)} \Gamma\\
     &\quad + \frac{d_{2,2}}{n-1} \Big( \pp \otimes \Omega\otimes \Omega + \Omega \otimes \pp \otimes \Omega + \Omega \otimes \Omega\otimes \pp\\
     &\quad + [\pp\otimes \O \otimes \pp \otimes \O]_{:24} + [\pp \otimes \O \otimes \O \otimes \pp]_{:25} + [\O \otimes \pp \otimes \O \otimes \pp]_{:35}
    \Big) \label{eq:H4inlemma}\\
H_4^r(\eta,\O,\O^\perp) & = \frac{d_{2,4}}{(n-1)(n+1)} T + \frac{d_{4,2}}{n-1}\Big(\pp\otimes \O \otimes\O\otimes\O
    + [\pp \otimes \O\otimes\pp\otimes\O\otimes \O]_{:24}\\
    &\quad+ [\pp\otimes\O\otimes\O\otimes\pp\otimes\O]_{:25} + [\pp \otimes \O\otimes\O\otimes\O\otimes\pp]_{:26}\Big), \label{eq:H4rinlemma}
\end{align}
Where $\Gamma$ is symmetric order-4 tensor defined by
\begin{align}\label{eq:Gamma}
	\Gamma = 3 Sym (\proopp \otimes \proopp ). 
\end{align}
In Cartesian coordinates $\Gamma$ is given by
\begin{align*}
	\Gamma_{ijkl}= (\proopp)_{ij} (\proopp)_{kl} + (\proopp)_{ik}(\proopp)_{jl} + (\proopp)_{il}(\proopp)_{jk}.
\end{align*}
Also, $T$ is a 5-tensor with components
\begin{align} \label{eq:defT}
    T_{ijklp}= \tilde S_{ijkl}\Omega_p + \tilde S_{ijkp}\Omega_l + \tilde S_{ijpl}\Omega_k + \tilde S_{ipkl} \Omega_j
\end{align}
where 
 \begin{align} 
        \tilde S_{iiii} = 3\Gamma_{iiii}, \quad
        \tilde S_{iijj} = \Gamma_{iijj}, \quad
        \tilde S_{ijij} = \Gamma_{ijij}, \quad
        \tilde S_{ijji} = \Gamma_{ijji},
    \end{align}
 for   $i,j \in \{1,...,n\}$ such that $i \neq j$, otherwise $\tilde S_{ijkl}$ is zero.
 Above we used the notation $$[\pp\otimes\O\otimes \pp]_{:24}$$ to denote a 3-tensor whose $(i,j,k)$ component is given by 
\begin{align*}
   \Big( [\pp\otimes \O\otimes \pp]_{:24} \Big)_{ijk}= \sum_{p=1}^n (\pp)_{ip} \O_j (\pp)_{pk},
\end{align*}
i.e., in this case we contract the full tensor $\pp\otimes\O\otimes \pp$, which is a 5-tensor, by summing over the second and the forth indices (that is why we have the sub-index notation $:24$). One defines the other contractions analogously.
\end{prp} The proof of the Proposition \ref{lem:Hdimension3} is given in Section \ref{sec:termBf}.

Moreover, from Theorem \ref{th:macroscopic limit} we can deduce the following:

\begin{cor}[Isotropic regime] 
	Let $n \in \{2,3\}$ and $a \in (0, 2]$. Suppose that $f^\eps \to \rho$ as $\eps \to 0$, then it holds that 
	\begin{align} \label{eq:rhowithzeroS2}
		\partial_t \rho - D_x \Delta_x \rho + \mu \lp  \frac{\chi^2}{n} - 1\rp \nabla_x\cdot(\rho\nabla_x\rho)=0.
	\end{align}
\end{cor}
\begin{proof} The proof of the corollary directly follows the steps of the proof of Theorem \ref{th:macroscopic limit} to obtain the equation for $\rho$.
\end{proof} 
The proof of Theorem \ref{th:macroscopic limit} is postponed to Section \ref{sec:proof}. Next, we comment on Theorem \ref{th:macroscopic limit}, particularly on Equations \eqref{eq:macro_rho} and \eqref{eq:macro_Omega}.

\subsection{Comments on the continuum equations} \label{ssec:interpretation}
In this section we discuss the macroscopic equations for the mass density and the mean-nematic direction \eqref{eq:macro_rho} and \eqref{eq:macro_Omega}, respectively. 

\subsubsection{The equation for the mass density}

Equation \eqref{eq:macro_rho} is a diffusion-type equation in divergence form and hence the total mass $\int_{\R^n} \rho \d x$ is preserved over time. The equation is formed by the sum of a linear diffusion and a non-linear diffusion term.
The purely diffusive term with the diffusion constant $D_x$ arises from the Brownian motion at the particle level. The second term on the right-hand-side resembles a porous medium equation, but with a diffusion coefficient $K$ that depends also on the density $\rho$. Numerical experiments indicate that $K(\eta(\rho))$, as given in Equation \eqref{eq:Ceta_coeff_porousmedium}, is non-negative, see Remark \ref{rem:K_non_negative}.

\paragraph{Effect of particle anisotropy.} Interestingly, the equation for the mass density  is independent of the mean-nematic direction $\Omega$. However, the anisotropy of the particles modifies Equation \eqref{eq:macro_rho}. To understand the effect of particle anisotropy, characterised by $\chi$ and defined in \eqref{eq:def-chi}, we first consider spherical particles. For spherical particles, $\chi=0$, and the potential $V_b$ becomes isotropic, i.e.,
\begin{align*}
V_b(x_2-x_1)= (4\pi)^{-n/2}\exp\left(-\frac{|x_2-x_1|^2}{2d^2}\right).
\end{align*}
In this case, particles' positions and directions are completely decoupled in the discrete system \eqref{eq:discrete_system} since $\nabla_u V_b=0$ and $\nabla_x V_b$ is independent of $u$. Moreover, one straightforwardly obtains the following equation for $\rho$ (by directly integrating equation \eqref{eq:dimless_kinetic_eq} and noting that  $W_f=\rho$ in this case):
\begin{align} \label{eq:isotropic_equation}
	\partial_t \rho = D_x\Delta_x \rho + \mu \nabla_x\cdot \lp \rho \nabla_x \rho \rp, \qquad \mbox{for }\quad \chi=0, 
\end{align} which is the porous medium equation with linear diffusion. Notice that it corresponds to Equation \eqref{eq:rhowithzeroS2}, where we assumed that $f^\eps\to \rho$ as $\eps \to 0$. Therefore, the difference between the isotropic and the anisotropic cases is encapsulated in the non-linear diffusion coefficient $K(\eta) \neq 1$.

Now, notice that
\begin{align} \label{eq:eta_and_S_2}
	K(\eta)
 \leq 1-\frac{\chi^2}{n},
\end{align}
since $S_2(\eta)\in (0,1)$ and $\eta$ is non-decreasing (see also Figure \ref{fig:eta_and_S_2}).

\begin{remark}[Non-negativity of the diffusion coefficient $K$] \label{rem:K_non_negative}
\begin{figure}[ht!] 
    \centering
\begin{subfigure}[b]{0.45\textwidth}
         \centering
         \includegraphics[width=1\linewidth]
         {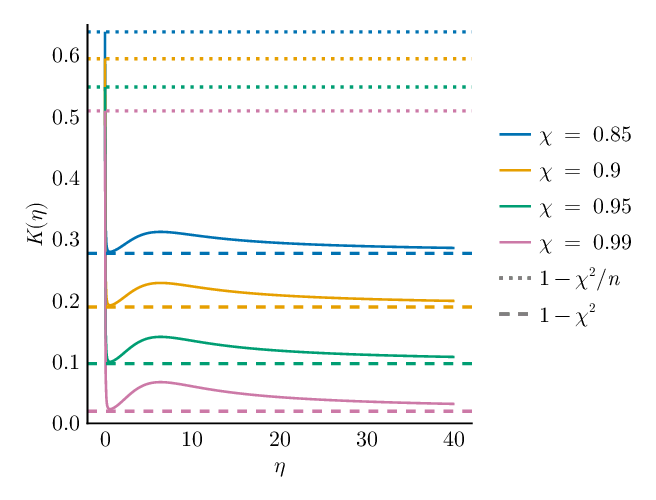}
         \caption{Evolution of $K(\eta)$ with respect to $\eta$.}
         \label{fig:K_bounds_a}
     \end{subfigure} 
    \begin{subfigure}[b]{0.45\textwidth}
         \centering
         \includegraphics[width=1\linewidth]
         {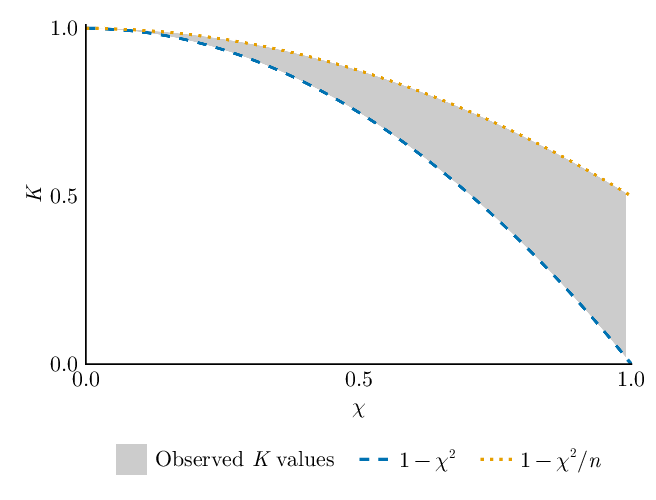}
         \caption{Range of $K(\eta)$ with respect to $\chi$.}
         \label{fig:K_bounds_b}
     \end{subfigure}
    \caption{(a) Time evolution of the diffusion coefficient $K(\eta)$ with respect to $\eta$ with different $\chi$ values. The dotted lines and the dashed lines correspond to $1-\chi^2/n$ (upper bound for $K$) and $1-\chi^2$ (lower bound for $K$), respectively, for the matching color $\chi$ values. (b) The range of $K(\eta)$ for $\chi \in (0,1)$ and $\sigma \in [2^{-8}, 2^{8}]$. The dotted and the dashed lines show the upper and the (numerical) lower bound for $K$, respectively.}
    \label{fig:K_bounds}
\end{figure}
Even though, we do not have an analytical guarantee that $K$ is non-negative, numerical simulations suggest that $K$ actually has a lower bound $1-\chi^2 \geq 0$. Figure \ref{fig:K_bounds_a} shows the evolution of $K$ as $\eta \geq 0$ varies for different values of $\chi$. The dashed lines in matching colors represent the numerical lower bound for $K$ and the dotted lines display the upper bound $1-\frac{\chi^2}{n}$ for $K$. Moreover, in Figure \ref{fig:K_bounds_b}, the gray area shows the range of $K$ values with respect to $\chi$ when other parameters, i.e. $\sigma$, also vary. Here, we can numerically further confirm that the lower bound for observed $K$ values is $1-\chi^2$, denoted by the blue dashed line. Note also that, in dimension $n=3$, $\chi$ can take negative values, i.e., $\chi \in [-1,1]$. However, this does not change the value of $K$ since only $\chi^2$ is involved in the computation of $K$.  That is why we only display $\chi \in[0,1]$ in Figure \ref{fig:K_bounds_b} and naturally it would be symmetric for $\chi \in[-1,0]$. Moreover, our numerical experiments include how $K$ changes with respect to $\rho$ for various $\chi \in (0,1)$ values, see Figure \ref{fig:K_bounds_2}, where, again, $K$ remains non-negative.
\end{remark}

Assuming that $K\geq 0$, the particle anisotropy contributes to a slowdown in the non-linear diffusion. In dimension $n=3$, when the particles are rods ($\chi=1$) or flat disks ($\chi=-1$) we have that $K(\eta) \leq 2/3$; thus, the original speed of diffusion $\mu$ is reduced to at least two-thirds of the anisotropic case. One may hypothesize that this effect is due to anisotropic particles being able to pack and occupy less volume when aligned, so they produce less diffusion. Indeed, our numerical investigations regarding the effect of particle anisotropy on the global alingment also confirm that particles with higher anisotropy reach total global alignment faster. We observe that the particles which are closer to perfect circle in shape ($\chi \approx 0$) and low in density do not reach alignment by the end of the simulation time, whereas particles which are more elongated ellipses ($\chi \approx 1$) with higher densities align globally in a shorter time frame, see, e.g., Figure \ref{fig:vary_l_d_2}. We refer to Section \ref{ssec:sim_results} for a detailed discussion.

\subsubsection{The equation for the mean-nematic direction}

First, we remark that the equation for the mean-nematic direction $\Omega$, Equation \eqref{eq:macro_Omega}, is non-conservative. It can be verified that $|\Omega|(t)$ remains constant over time, as all the terms lie within the space spanned by $\Omega^\perp$. Consequently, $\Omega$ stays on the sphere at all times. 

The term $\Pi_2$ corresponds to transport of $\Omega$ in the direction $\nabla_x\rho$. However, the sign of $\Pi_2(\rho)$ is not immediately clear, or it is uncertain if it changes sign during the dynamics. Notably, $\Pi_2$ would vanish if $\mu=0$, indicating that this term arises from the repulsive potential acting on particle positions. 

The last term on the left-hand side of \eqref{eq:macro_Omega} corresponds to diffusion in the direction $\Omega$. The projection term $P_{\Omega^{\perp}}$ ensures that $\Omega$ remains on the sphere. One can check that when particles are spherical, there is no equation for $\Omega$, which is expected since the direction $\Omega$ does not influence the dynamics in this case. If $\mu=0$, the diffusion constant reduces to $D_x$, meaning that, in this case, the diffusion  originates only from the Brownian motion of the positions of the particles. Intriguingly, the constant $\mu \sigma$ introduces an anti-diffusion effect.

\paragraph{The case $a\in(0,2)$.} When $a\in(0,2)$, the right-hand side of \eqref{eq:macro_Omega} vanishes. In this case, the dynamics for $\rho$ and $\Omega$ are identical for prolate and oblate particles with the same values of $\chi^2$. Consequently, at the macroscopic level, the equations do not differentiate between these two types of particles.

\subsubsection{The case \texorpdfstring{$a=2$}{a=2}} \label{ssec:Pi_3}

When $a=2$, the term $\Pi_3$, given by Equation \eqref{def:Pi3}, does not vanish. This term gives the effect of the repulsive potential on the directionality of the particles. Specifically, from Lemma \ref{lem:rescaling the potential}, the potential $V^\eps_f$ expands as
$$V_f^\eps = W_f +\eps^2 B_f +\mathcal{O}(\eps^3),$$
where $W_f$ is a nematic-alignment force, and the term $B_f$ - which generates $\Pi_3$ - arises from the specific form of the repulsive potential $V_b$ considered. The term $\Pi_3$ involves either second-order differential operators or products of first-order differential operators.

Interpreting $\Pi_3$ is not straightforward, but it is clearly distinct from the other terms in Equation \eqref{eq:macro_Omega}. Notably, $\Pi_3$ is the only term where the distinction between oblate ($\ell<d$) and prolate ($\ell>d$) particles matter. Specifically, oblate and prolate particles exhibit opposite signs for the first and  last two terms of $\Pi_3$. If the diffusion constants are positive for prolate particles, they will be negative for oblate particles, and vice versa. As a result, the qualitative behavior of oblate and prolate particles will differ significantly. 

Notice that the second term in $\Pi_3$ is solely influenced by the length $d$. If $d=0$ (i.e., if the particles are rod-shaped), this term disappears.

To simplify the interpretation of $\Pi_3$, in the next two paragraphs we focus only on dimension $n=2$.

\paragraph{Anisotropic equilibrium ($d_{2,0}\neq d_{0,2}$).} One can expand the contractions in $\Pi_3$ further considering expressions \eqref{eq:H2d2}-\eqref{eq:H4rd2}. For example, for the second term in \eqref{def:Pi3} we have that:
\begin{align*}
	[H_2^r: (\Delta_x(\rho H_2))]_{[2,3:1,2]} &= 2d_{2,2}d_{2,0}\Omega^\perp \left( \rho(\Omega^\perp\cdot \Delta_x\Omega) + 2\Omega^\perp \cdot [(\nabla_x \rho \cdot \nabla_x)\Omega]\right)\\
	&\, + 2d_{2,2}d_{0,2}\Omega^\perp\left(\rho(\Omega \cdot \Delta_x\Omega^\perp)+ 2\Omega \cdot[(\nabla_x \rho \cdot \nabla_x)\Omega^\perp]  \right).
\end{align*}
For these equations, it is critical that $d_{2,0}\neq d_{0,2}$, since, if they were the same, then $H_2$ would be equal to the identity and this would imply that 
\begin{align*}
	[H^r_2:(\Delta\rho \Id)]_{[2,3:1,2]}=0,
\end{align*} since $(\Omega\otimes \Omega^\perp):\Id=\Omega\cdot \Omega^\perp=0$. Moreover, if $d_{2,0}= d_{0,2}$, then the third term of \eqref{def:Pi3} would become:
\begin{align}
[H_4^r: D_x^2\otimes \rho\Id]_{[2,3,4,5:1,2,3,4]} &= d_{2,0}(d_{2,4}+d_{4,2})\Omega^\perp[(\Omega\otimes \Omega^\perp+\Omega^\perp \otimes \Omega): (D_x^2\rho)].
\end{align}

Therefore, the fact that $d_{2,0}\neq d_{0,2}$ creates an anisotropy that produces derivatives in $\Omega$ and $\Omega^\perp$. If $\GetaA$ was the uniform distribution (i.e., if $\ell=d$), then we would have that $d_{2,0} = d_{0,2}$. Thus, this effect is linked to the anisotropy of the particles. 

\paragraph{Cross-diffusion term.} The term of $\Pi_3$ containing the second derivatives of $\rho$ is
\begin{align}
	b(\rho) \rho\Omega^\perp [(\Omega\otimes \Omega^\perp + \Omega^\perp \otimes \Omega): D^2_x\rho],
\end{align} where 
\begin{align*}
b(\rho):= \frac{n-1}{8\eta(\rho) c_{1,2}}(\ell^2-d^2)  \left ( d_{2,2}
	-\chi^2(d_{4,2}d_{2,0}+d_{2,4}d_{0,2})	
        -2\chi^2  d^2_{2,2}   \right ).
\end{align*}
This term makes the system \eqref{eq:macro_rho}-\eqref{eq:macro_Omega} of a cross-diffusion type.
In dimension $n=3$, it is clear that if the diffusion coefficient $b(\rho)$ is positive for prolate particles then it is negative for oblate particles (and vice versa). This will particularly complicate the study of the well-posedness of the equations, which we leave for future work.

\subsection{Parameter regime for the validity of the continuum equations} \label{ssec:well_posedness}
It remains an open question to determine for which parameter regime the macroscopic equations \eqref{eq:macro_rho}-\eqref{eq:macro_Omega} are well-posed. In particular,  the consistency relation \eqref{eq:consistency_relation} imposes a major constraint on the dynamics. Particularly, when $\rho<\rho_*$, the macroscopic equations lose their validity. In this section, we illustrate this effect with a couple of examples.

\paragraph{Diffusion term in $\Omega$.} For the well-posedness of the equations, we require the diffusion constant in $\Omega$, given by $\mu(\nu-\sigma)$, (see the left-hand side of \eqref{eq:macro_Omega}), to be positive. This is equivalent to the condition:
\begin{align*}
	\frac{\mu}{D_x}< \frac{\lambda}{D_u}.
\end{align*}
This condition establishes a balance between the noise intensities $D_x,\, D_u$ and the intensity of the alignment interactions $\mu,\, \lambda$. If $\mu=\lambda$, then we must have $D_u< D_x$ (i.e., the noise intensity in the directions must be smaller than in the positions). If $D_x=D_u$, then $\mu<\lambda$ must hold (i.e., the intensity of the potential must be stronger for the directions than for the positions). Additionally, for well-posedness, it is required that $D_x>0$. These parameter constraints suggest that there are certain aspects of the microscopic dynamics that the macroscopic equations fail to capture. Moreover, in Section \ref{ssec:particle_sims}, we present some numerical tests on this condition at the microscopic level, simulating the stochastic particle system \eqref{eq:discrete_system}. The results indicate that when the continuum equations are ill-posed, i.e., when $\sigma \geq \nu$, the particles do not reach any type of alignment. This can be observed particularly in Figures \ref{fig:vary_DxDu} and \ref{fig:vary_lambda_mu}. For further details, we refer the reader to Section \ref{ssec:particle_sims}. 

\paragraph{The critical case $\nu=\sigma$.} Let us consider the critical case where $\nu=\sigma$. In this case the diffusive term in $\Omega$ vanishes. We explore this scenario in more detail.\\
Notice that we can rewrite the kinetic equation \eqref{eq:rescaled_kinetic} (for $a \in (0,2)$) as:
\begin{multline}\label{eq:particular_case_rescaled}
	\p_t f^\eps - D_x \nabla_x \cdot \lp \nabla_x f^\eps - \frac{\chi^2\mu}{D_x}f^\eps \nabla_x \lp u^T (\rho_f Q_f )u \rp  \rp + \mu \lp \frac{\chi^2}{n}-1\rp \nabla_x \cdot ((\nabla_x \rho_{f^\eps}) f^\eps) \\ = \frac{D_u }{\eps}C(f^\eps)+ \mc O (\eps^2),
\end{multline} where, we recall,
\begin{align*}
	C(f) =\nabla_u \cdot \lp \nabla_u f - \frac{\chi^2\lambda}{D_u}f\nabla_u \lp u^T (\rho_f Q_f )u \rp \rp.
\end{align*}
Notice that the second term in \eqref{eq:particular_case_rescaled} has the same shape as the operator $C$ except that the derivatives are in $x$ and the parameter values differ.

If $\nu=\sigma$, i.e., $ \frac{D_x}{\mu} = \frac{D_u}{\lambda}$, the second term in \eqref{eq:particular_case_rescaled} can be rewritten as
\begin{align}\label{eq:mu_equal_sigma}
	- D_x \nabla_x \cdot \left(G_{\alpha \rho_f Q_f}\nabla_x\left(\frac{f^\eps}{G_{\alpha \rho_f Q_f}}\right) \right), \quad \text{with   } \alpha = \frac{\chi^2}{\sigma}.
\end{align}
By Lemma \ref{lem:stable_equilibria}, when $f^\eps \to \rho \GetaA$, we have that  $G_{\alpha \rho_f Q_f}\to \GetaA$ and in the limit \eqref{eq:mu_equal_sigma} becomes 
\begin{align*}
	-D_x\nabla_x\cdot \left(\GetaA \nabla_x \rho) \right). 
\end{align*} 
Integrating \eqref{eq:particular_case_rescaled}, and using the fact that $\int \GetaA \d u =1$, we obtain Equation \eqref{eq:rhowithzeroS2}, which is the equation obtained in the isotropic regime corresponding to $f^\eps \to \rho$ as $\eps\to 0$.  A possible explanation for this is that when $\mu = \sigma$ the density $\rho$ declines rapidly below the critical density $\rho_*$ due to the diffusive processes involved.  

\medskip
In conclusion, we emphasize two main takeaways on the validity of the continuum equations. Firstly, the macroscopic limit does not hold for arbitrary parameter values. Secondly, the macroscopic equation ceases to be valid when $\rho < \rho_*$ (at the points where Equation \eqref{eq:rhowithzeroS2} holds and we enter the isotropic regime). However, we have no prior information on the specific points in space or the times at which this may occur.

\medskip
We finish with a remark concerning the validity of the continuum equations for higher dimensions.

\begin{remark}[Higher dimensions]
	The macroscopic limit presented in this paper only holds for $n\in\{2,3\}$ since in Lemma \ref{lem:stable_equilibria} we only obtain stable equilibria for these dimensions. Existence of stable equilibria for $n\geq4$ is an open problem. However, if Lemma \ref{lem:stable_equilibria} holds for $n \geq 4$, we could conjecture that Theorem \ref{th:macroscopic limit} would also hold for $n \geq 4$.
\end{remark}

\section{Repulsive potentials and particle simulations}

\label{sec:repulsion_general}

In this section, we comment on the different Gaussian-type repulsive potentials, motivate our choice \eqref{eq:weighted_Gaussian_potential} and present numerical simulations for particles following the dynamics \eqref{eq:discrete_system_X}-\eqref{eq:discrete_system_u}.

\subsection{Gaussian-type anisotropic repulsive potentials}
\label{ssec:potentials}

Earlier, we described the binary interactions between identical particles via an anisotropic potential $V_b$ given by Equation \eqref{eq:potentialG}. We comment on two particular choices of the scaling factor $b \geq 0$ in \eqref{eq:potentialG}.

\paragraph{The Berne-Pechukas potential.}
One of the first soft anisotropic potentials that appeared in the literature is that of a soft repulsive potential between two spheroids using the Gaussian potential. It was first introduced by Berne and Pechukas in \cite{BP72}. In \cite{BP72}, the scaling factor  $b(u_1, u_2)$ is given by
\begin{equation} \label{eq:rescaling_BP}
	b_{\text{BP}}(u_1, u_2):= (1-\chi^2(u_1\cdot u_2)^2)^{-1/2}=\mbox{det}(\Sigma)^{-1/2},
\end{equation}
where $\chi$ is defined by \eqref{eq:def-chi}. This is a very natural choice for the scaling factor $b$, as it guarantees that $V_b$ is a multivariate Gaussian distribution in the variable $R$ with a covariance matrix given by $ \Sigma/2$.

\paragraph{The non-scaled Gaussian potential.}
Berne and Pechukas in \cite{BP72} considered a scaling factor $b$ such that the potential \eqref{eq:potentialG} is normalized, i.e., 
$\int_{\R^d} V_{b_{\text{BP}}}(u_1,u_2,R)\, \d R=1.$
However, one could choose to work directly with the potential without normalization, i.e., consider the scaling factor 
\begin{align} \label{eq:non-rescaled_factor}
	b_{\text{NR}}(u_1,u_2):=(4\pi)^{n/2}.
\end{align}
Notice that both, the potential $V_{b_{\text{BP}}}$ and the potential without normalization $V_{b_{\text{NR}}}$, have the same level sets in space (up to scaling). The level sets in directions, however, are different and this has important consequences. 

\subsubsection{Motivation for the interaction potential}  \label{ssec:motivation_potential}
We focus on the non-scaled Gaussian potential for two reasons. Firstly, the Berne-Pechukas potential does not produce nematic alignment at first order, see below item (ii). Secondly, we want to avoid the modification in the directions introduced by the scaling. However, considering a system interacting via this potential is mathematically challenging to study, and, therefore, we consider a modified version of the non-scaled Gaussian potential, retaining the similar properties, this was explained in Remark \ref{rem:Berne_Pechukas}. Instead of \eqref{eq:non-rescaled_factor}, we work with \eqref{eq:weighted_Gaussian_potential}, as introduced earlier. This simplification resembles to the simplification done when considering the Maier-Saupe potential $b_{\text{MS}}(u,u'):=(u\cdot u')^2$ instead of the Onsager potential $b_{\text{O}}(u,u'):=|u\cdot u'|$, see, \cite{WH08}. Both potentials have the same minimizer, but $b_{\text{MS}}$ is easier to study mathematically. 

Moreover, using \eqref{eq:weighted_Gaussian_potential}, we then observe that the leading term of the expansion \eqref{eq:rescaleV} is given by
\begin{align}\label{eq:potentialW}
	W_f(t, x_1,u_1) := \ints\lp 1-\chi^2 (u_1\cdot u_2)^2\rp f(t, x_1,u_2) \d u_2. 
\end{align}
The function inside the integrand $\varphi (u_1,u_2):=\lp 1-\chi^2 (u_1\cdot u_2)^2\rp$ resembles that of the Maier-Saupe potential \cite{WH08} and motivates our choice for the rescaling factor $b$. In particular,
\begin{itemize}
	\item[(i).] if we consider $b \equiv 1$ (i.e., without scaling of the potential), we obtain
	\begin{align}
		\tilde W_f(t, x_1,u_1) := \ints \lp 1-\chi^2 (u_1\cdot u_2)^2\rp ^{1/2} f(t, x_1,u_2) \d u_2, 
	\end{align}
	with $\tilde\varphi(u_1,u_2):=\lp 1-\chi^2 (u_1\cdot u_2)^2\rp ^{1/2}$. Even though the potential with $\varphi$ and the potential with $\tilde{\varphi}$ have the same minimizers in terms of $u_1,u_2$, $W_f$ is more straightforward to analyze since we can rewrite it as
	\begin{align*}
		W_f(t,x_1,u_1)=u_1^T\lp\ints ( \mathrm{Id}- \chi^2(u_2\otimes u_2) ) f(t,x_1,u_2) \d u_2\rp u_1,
	\end{align*}
i.e.,  we can decouple the variables $u_1$ and $u_2$ and recast $W_f$ in terms of $Q_f$ (see Equations \eqref{eq:Qdef} and \eqref{def:W_f}). As previously mentioned, a similar simplification is done when, instead of considering the Onsager potential $b_{\text{O}}(u_1,u_2)=|u_1\cdot u_2|$, one considers the Maier-Saupe potential $b_{\text{MS}}(u_1,u_2)=(u_1\cdot u_2)^2$, see \cite{WH08} for more details.
\item[(ii).] if we consider the Berne-Pechukas scaling $b_{\text{BP}}$ in \eqref{eq:rescaling_BP}, then the corresponding potential takes the form
\begin{align}
	W_f^{\text{BP}}(t,x_1)= \ints f(t, x_1,u_2) \d u_2 =  \rho(t, x_1).
 \label{eq:potential_BP}
\end{align}
Since the total integral of $V_{b_{\text{BP}}}$ is one, we do not obtain a nematic alignment potential at the leading order.
\end{itemize}
In Section \ref{ssec:comparison_potentials}, we provide a numerical comparison of stochastic particle systems interacting via these different potentials. Indeed, we observe that particles interacting via the Berne-Pechukas potential do not align, see Figure \ref{fig:sim_vary_potential}.

\subsection{Numerical simulations}
\label{ssec:particle_sims}

This section is dedicated to the numerical simulations of the stochastic particle system \eqref{eq:discrete_system}, particularly to compare the particle-level dynamics at longer times with our analytical predictions on the macroscopic dynamics. 

We start with the study of how the central assumption for well-posedness \eqref{eq:intro_condition_wellposedness} at the macroscopic level affects the particle dynamics. Subsequently, we investigate the impact of the anisotropy of the particles (measured by $\chi$ defined in \eqref{eq:def-chi}) and the particle density on the particle interaction dynamics.  We finish the numerical simulations section with a comparison between the different types of interaction potentials mentioned in Section \ref{ssec:potentials}.

The source code for the particle simulations is provided at 
\begin{center}
    \url{https://github.com/cwytrzens/anisotropic_particles}.
\end{center}

\subsubsection{Numerical implementation}
We simulate the discrete particle dynamics in dimension $n=2$, governed by the stochastic differential equations \eqref{eq:discrete_system_X} and \eqref{eq:discrete_system_u} on a periodic domain. The simulations use an 
explicit Runge-Kutta Milstein method of strong order 1 with adaptive time stepping. This numerical method is part of the Julia package \texttt{StochasticDiffEq.jl}, see \cite{RN17}. Since we want to study a wide range of model parameters, it is crucial to use an adaptive time-stepping method, as the stiffness of the system strongly depends on the model parameters and differs across the parameter range that we investigate. We use default tolerance parameters \texttt{reltol} $= 10^{-2}$ and \texttt{abstol} $= 10^{-2}$. 
Moreover, in order to allow faster numerical evaluations of the interaction forces, we truncate the potential $V_{b}$ as defined in Equation \eqref{eq:potentialG} for $R = | X_j - X_i | \geq 8 \max(\ell,d)$. 

For the particle simulations, we consider a square domain of size $[0,L_x]\times[0,L_y]$ with $L_x=L_y=100$ and fix the number of particles at $N= 10^5$. The initial positions $X_i$ and directions $u_i$ of the particles are taken uniform and random for $i=1,..., N$. Moreover, we set $t_{\text{end}}=1.5 \times 10^5$, the time at which the simulations stop.

We perform an extensive parameter study for each remaining parameter pair, namely, the diffusion constants in space $D_x$ and direction $D_u$; the strength factors of the repulsive force, $\mu$ and $\lambda$; the anisotropy parameter $\chi$, and the particle density $\bar{\rho}$ which we define as  
\begin{equation}
    \bar{\rho} \coloneqq \frac{\pi \ell d N}{L_x L_y} \label{eq:def-rhobar},
\end{equation} 
and it represents the ratio of the total volume of all particles to the domain size.  The parameters $\chi$ and $\bar{\rho}$ determine $d$
 and $\ell$ uniquely (by solving \eqref{eq:def-chi} and \eqref{eq:def-rhobar} for $d$ and $\ell$).

\subsubsection{Parameter study}

For the parameter study, we focus on capturing the regime of parameter values at which the global alignment is reached versus those at which the particles fail to align globally. As the dynamics depend on many parameters and the balance between them, capturing all interesting parameter combinations is very challenging; thus, out of the scope of this numerical study. To measure alignment, we compute the order parameter $\gamma_f$ which is given in \eqref{def:deg_align}. If not stated otherwise, all the simulations use the default parameters given in Table \ref{tab:model_params}.

\begin{table}[ht!]
    \centering
    \begin{tabular}{cc}
        \toprule
         Parameter & Value \\ 
         \midrule
         $\chi$ & $0.9$ \\ 
         $\overline{\rho}$ & $1$ \\
         $D_x$ & $2^{-4}$ \\ 
         $D_u$ & $2^{-11}$ \\
         $\lambda$ & $2^8$ \\ 
         $\mu$ & $2^{13}$ \\ 
         $N$ & $10^5$ \\
         $L_x, L_y$ & $100$ \\ 
         $t_{\mathrm{end}}$ & $1.5 \times 10^5$ \\
         \bottomrule
    \end{tabular}
    \caption{Default model parameters for the simulations of the particle system.}
    \label{tab:model_params}
\end{table}

To capture the transition to alignment, we typically vary one or two parameters while keeping the others fixed. This approach allows for a direct comparison between our numerical results and the theoretical prediction that the condition $\sigma < \nu$ is necessary for well-posedness of the continuum equations, see \eqref{eq:intro_condition_wellposedness}. Recalling that $\sigma = \frac{D_u}{\lambda}$ and $\nu = \frac{D_x}{\mu}$, we study the trajectories of the order parameter $\gamma_f$ by varying $D_u$ relative to $D_x$, and $\lambda$ relative to $\mu$, while keeping the remaining parameters constant.

\medskip 

Another primary feature of our particle system is the global alignment  induced by  the anisotropy of the particles. To investigate this feature, we conduct a parameter study of $\chi$ versus $\bar \rho$. Specifically, we study the impact of the anisotropy parameter $\chi$ on global alignment by defining the parameters $\ell$ and $d$ implicitly in terms of $\chi$ and $\bar{\rho}$. 

\medskip 
For each parameter pair $(D_x, D_u)$, $(\lambda, \mu)$, and $(\chi, \bar{\rho})$, we present three plots:

\begin{enumerate}[label=(\alph*)]
    \item The expected trajectories of the mean global alignment $\gamma_f$, obtained by varying one parameter while keeping the others constant.
    \item A similar plot for the second parameter of the pair.
    \item A heatmap of the mean order parameter $\gamma_f$ at the final simulation time $t_{\text{end}}$, obtained by varying both parameters simultaneously. The heatmap highlights the transition to alignment with the transition front ($\sigma = \nu$) marked by a red line. 
\end{enumerate}

For each parameter combination, we collect $8$ simulation samples to estimate the trajectory of the mean global alignment, where the matching color bands show the standard deviation.

\subsubsection{Simulation results} \label{ssec:sim_results}

In this section, we present the results of our numerical experiments.

We start with the study of our first parameter pair, diffusion constants $D_x$ and $D_u$, presented in Figure \ref{fig:vary_DxDu}. Figure  \ref{fig:vary_Dx} displays the  trajectories of the order parameter $\gamma_f$ over time for different values of $D_x = 2^{-k}$ for $k = 2,\dots,7$ and $D_u$ fixed at $0$. Hence, there is no noise present in the particle directions $u$. We observe that the smaller the value of $D_x$, the slower the emergence of the global alignment. We fix $D_x=2^{-4}$ as the default value and look at various $D_u$ values in Figure \ref{fig:vary_Du}. 

\begin{figure}[ht!]
     \centering
     \begin{subfigure}[b]{0.32\textwidth}
         \centering
         \includegraphics[width=1.1\linewidth]
         {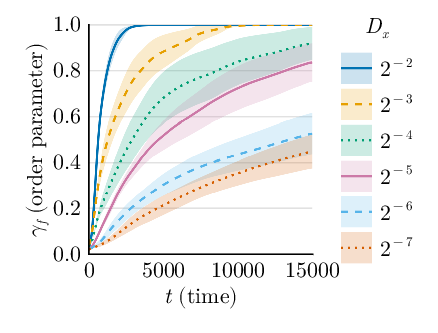}
         \caption{$D_u=0$ and $D_x$ varying}
         \label{fig:vary_Dx}
     \end{subfigure}
     \hfill
     \begin{subfigure}[b]{0.32\textwidth}
         \centering
         \includegraphics[width=1.1\linewidth]
         {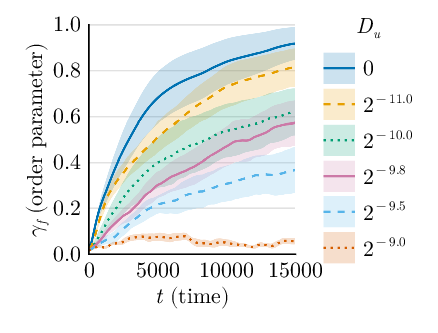}
         \caption{$D_x=2^{-4}$ and $D_u$ varying}
         \label{fig:vary_Du}
     \end{subfigure}
     \hfill
     \begin{subfigure}[b]{0.32\textwidth}
         \centering
         \includegraphics[width=1.1\linewidth]
         {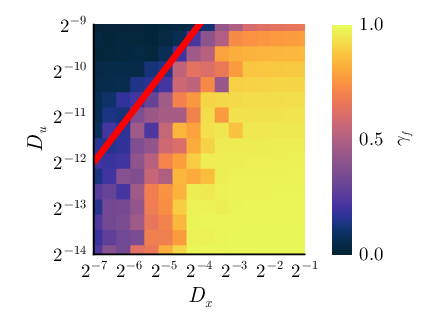}
         \caption{Heatmap of $\gamma_f$: $D_x$ vs. $D_u$.}
         \label{fig:heatmap_DxDu}
     \end{subfigure}
        \caption{
        Parameter study for $(D_x, D_u)$. Trajectories of the order parameter $\gamma_f$ are plotted for fixed $D_u=0$ and varying $D_x$ in Figure \ref{fig:vary_Dx}; for fixed $D_x=2^{-4}$ and varying $D_u$ in Figure \ref{fig:vary_Du}. Figure \ref{fig:heatmap_DxDu} displays a heatmap of $\gamma_f$ at time $t_{\text{end}}=1.5 \times 10^5$ while both $D_u$ and $D_x$ vary simultaneously. The red line depicts where $\nu=\sigma$.}
        \label{fig:vary_DxDu}
\end{figure}
We observe an opposite effect for $D_u$. In Figure \ref{fig:vary_Du}, varying $D_u$ shows that less angular noise leads to quicker alignment and too large noise might prevent reaching alignment all-together.

Finally, in Figure \ref{fig:heatmap_DxDu}, we display the order parameter $\gamma_f$ at the final simulation time $t_{\text{end}}$ as a heatmap of varying $D_x$ versus $D_u$. Here, we observe the parameter range leading to alignment and a transition along level sets of $D_x$ versus $D_u$. At the particle level, the heatmap can be interpreted as follows. Increasing the positional noise facilitates interactions among a greater variety of particles, as neighboring particles change quickly, which contributes to reaching global alignment. On the contrary, increasing the directional noise drives the system away from global alignment, which is, in particular, critical if particles only interact with their nearest neighbors. Consequently, we observe a lack of alignment in the top left corner of the heatmap, where positional noise is low and directional noise is high.

Note that the red line displays where $\nu=\frac{D_x}{\mu} = \frac{D_u}{\lambda} = \sigma$. This indicates the interface at which the macroscopic equation \eqref{eq:macro_Omega} for the mean-nematic direction $\Omega_f$ becomes ill-posed, see the discussion on well-posedness in Section \ref{ssec:well_posedness}. Crucially, the right-hand side of this red line is the parameter regime ($\sigma < \nu$) where the macroscopic equations remain well-posed, and it is exactly in this parameter regime that we observe the onset of particle alignment.

In Figure \ref{fig:vary_lambda_mu}, we investigate how the forces acting on positions $X_i$ and directions $u_i$ affect the alignment of particles. The strength factors of these forces are determined by the positive constants $\mu$ and $\lambda$ in front of the interaction terms in Equations \eqref{eq:discrete_system_X} and \eqref{eq:discrete_system_u}. 

\begin{figure}[ht!]
     \centering
     \begin{subfigure}[b]{0.32\textwidth}
         \centering
         \includegraphics[width=1.1\linewidth]{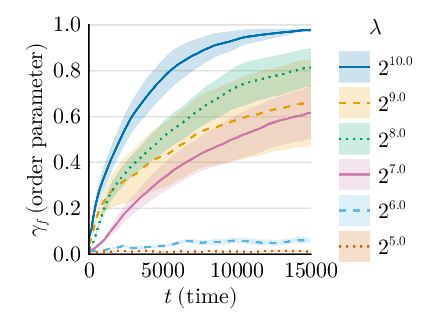}
         \caption{$\mu = 2^{13}$ and $\lambda$ varying.}
         \label{fig:vary_lambda}
     \end{subfigure}
     \hfill
     \begin{subfigure}[b]{0.32\textwidth}
         \centering
         \includegraphics[width=1.1\linewidth]{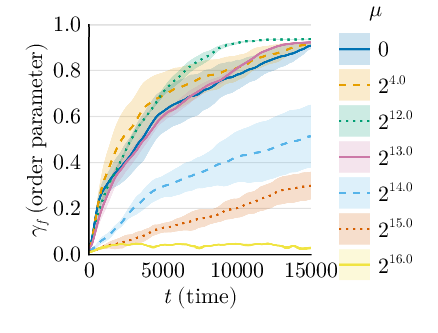}
         \caption{$\lambda = 2^{8}$ and $\mu$ varying.}
         \label{fig:vary_mu}
     \end{subfigure}
     \hfill
     \begin{subfigure}[b]{0.32\textwidth}
         \centering
         \includegraphics[width=1.1\linewidth]{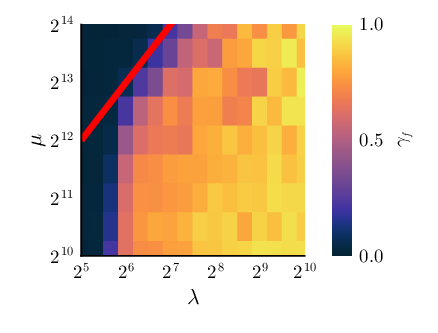}
         \caption{Heatmap of $\gamma_f$: $\lambda$ vs. $\mu$.}
         \label{fig:heatmap_lambda_mu}
     \end{subfigure}
        \caption{Parameter study for $(\lambda, \mu)$. 
        Trajectories of the order parameter $\gamma_f$ are plotted for fixed $\mu=2^{13}$ and varying $\lambda$ in Figure \ref{fig:vary_lambda}; for $\lambda=2^8$ and varying $\mu$ in Figure \ref{fig:vary_chi}. Figure \ref{fig:heatmap_lambda_mu} displays the heatmap of $\gamma_f$ at time $t_{\text{end}}=1.5 \times 10^5$ while both $\lambda$ and $\mu$ vary simultaneously. The red line depicts where $\nu=\sigma$.}
        \label{fig:vary_lambda_mu}
\end{figure}

Figure \ref{fig:vary_lambda} shows the trajectories of the order parameter $\gamma_f$ over time for fixed $\mu = 2^{13}$ and varying values of $\lambda$ from $2^5$ to $2^{10}$. We observe that for $\lambda\leq 2^{6}$, the system is in total chaos, with $\gamma_f$ close to $0$, indicating that the force acting on the particles' directions is not strong enough to drive the system towards alignment. Additionally, we deduce that as $\lambda$ increases, $\gamma_f$ appears to increase accordingly.

Next, we fix $\lambda=2^8$ and vary $\mu$ from $0$ to $2^{16}$ to study how the strength of the potential $V_b$ affects the alignment through the force acting on the particles' positions. In Figure \ref{fig:vary_mu}, we observe that $\mu$ must be sufficiently large to have a significant impact on the particle dynamics. As $\mu$ increases, the particles lose their alignment, see, e.g., when $\mu= 2^{16}$, $\gamma_f$ is very close to $0$. 

Lastly, we present a heatmap of the order parameter $\gamma_f$ in Figure \ref{fig:heatmap_lambda_mu}, similar to Figure \ref{fig:heatmap_DxDu}, but for $\lambda$ versus $\mu$. Similarly, the red line displays the interface of the analytical well-posedness condition $\nu = \sigma$, see Equation \eqref{eq:intro_condition_wellposedness}. Again, we observe that global alignment only occurs inside the parameter range where the macroscopic equations are well-posed ($\sigma<\nu$), located on the right-hand side of the red line.

As for our last parameter pair, in Figure \ref{fig:vary_chi_rhobar_trajectories} we study the effects of the anisotropy $\chi$ and the density of particles $\bar \rho$ on the global alignment. In Figure \ref{fig:vary_chi}, we present a plot of the trajectories of $\gamma_f$ when $\bar \rho = 1$ and the anisotropy parameter $\chi$ is varying from $0.6$ to $1$. We can clearly deduce that a certain degree of anisotropy is required for particles to reach global alignment. 

\begin{figure}[ht!]
	\centering
	\begin{subfigure}[b]{0.32\textwidth}
		\centering
		\includegraphics[width=1.1\linewidth]{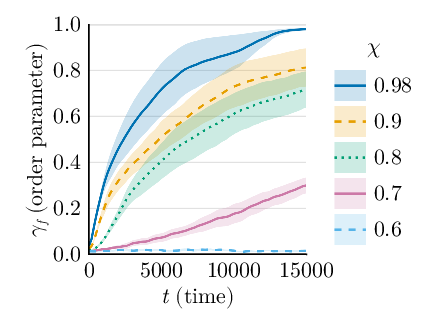}
		\caption{$\bar \rho =1$ and $\chi$ varying.}
		\label{fig:vary_chi}
	\end{subfigure}
	\hfill
	\begin{subfigure}[b]{0.32\textwidth}
		\centering
		\includegraphics[width=1.1\linewidth]{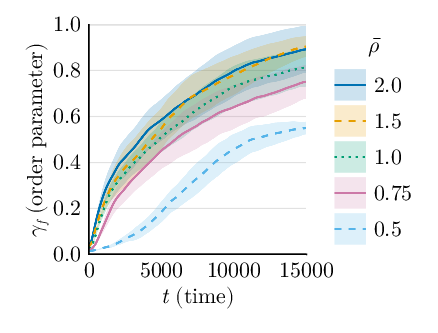}
		\caption{$\chi=0.9$ and $\bar{\rho}$ varying.}
		\label{fig:vary_rhobar}
	\end{subfigure}
	\hfill
	\begin{subfigure}[b]{0.32\textwidth}
		\centering
		\includegraphics[width=1.1\linewidth]{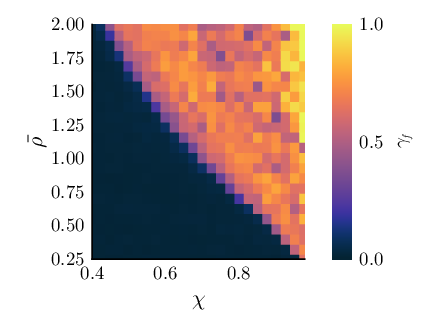}
		\caption{Heatmap of $\gamma_f$: $\chi$ vs. $\bar{\rho}$.}
		\label{fig:heatmap_chi_rhobar_1}
	\end{subfigure}
	\caption{Parameter study for $(\chi, \bar \rho)$. Trajectories of the order parameter $\gamma_f$ are plotted for fixed $\bar \rho = 1 $ and varying $\chi$ in Figure \ref{fig:vary_chi}; for fixed $\chi = 0.9$ and varying $\bar \rho$ in Figure \ref{fig:vary_rhobar}. Figure \ref{fig:vary_chi_rhobar_trajectories} displays the heatmap of $\gamma_f$ at time $t_{\text{end}}=1.5 \times 10^5$ while both $\chi$ and $\bar \rho$ vary simultaneously.}
	\label{fig:vary_chi_rhobar_trajectories}
\end{figure}

In Figure \ref{fig:vary_rhobar}, we fix $\chi=0.9$ and vary $\bar \rho$ from $0.5$ to $2$ (recall that $\bar \rho$ is computed using \eqref{eq:def-rhobar}) and plot the trajectories of $\gamma_f$. We observe that the particle density must also be sufficiently high for particles to align, otherwise, if the particles are too far apart, the repulsive potential is not strong enough to induce alignment. Figure \ref{fig:heatmap_chi_rhobar_1} presents a heatmap illustrating the effect of both $\chi$ and $\bar \rho$ on the transition towards a higher-order alignment. Here, we also observe a clear linear progression, moving from a chaotic state, where $\gamma_f$ is close to $0$, to a well-ordered system, characterized by global alignment with $\gamma_f\approx 1$.

\begin{figure}[ht!] 
	\centering
		\begin{tabular}[c]{m{8cm} m{7cm}}
			\includegraphics[width=\linewidth]{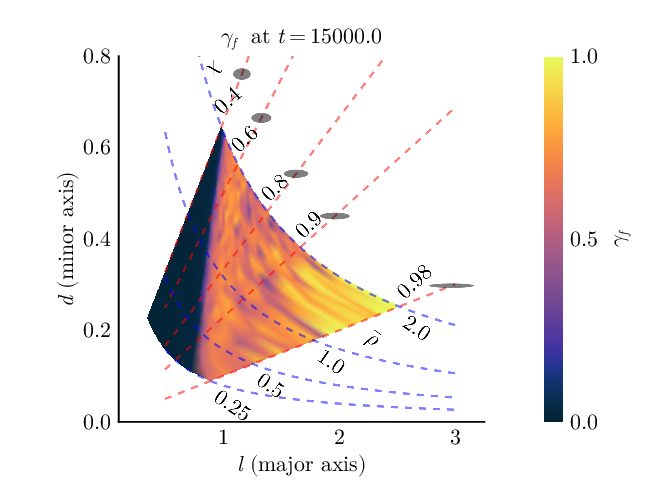} &
			\caption{Heatmap of $\gamma_f$: $\chi$ versus $\bar \rho$. The dashed red lines display the anisotropy parameter $\chi$ for the corresponding lenghts of the major and minor axes of the ellipses, $\ell$ and $d$ respectively. Ellipses with corresponding $\chi$ are pictured in gray. The dashed blue lines are the constant density curves $\bar \rho$ corresponding to $\ell$ and $d$, computed using \eqref{eq:def-rhobar}.}   	\label{fig:vary_l_d_2}
		\end{tabular}
\end{figure}

Moreover, to further investigate the effect of the shape of the particles on the global alignment of the system, we present in Figure \ref{fig:vary_l_d_2} another heatmap of $\gamma_f$ with the same dataset as in Figure \ref{fig:heatmap_chi_rhobar_1}. However, in Figure \ref{fig:vary_l_d_2} we compare the effect of varying the lengths of the major and the minor axes of the ellipses, i.e., $\ell$ and $d$, on global alignment rather than with respect to $\chi$ and $\bar{\rho}$. Notice that depending on $\ell$ and $d$, both $\chi$ and $\bar \rho$ vary across the heatmap accordingly. The red dashed lines display the constant $\chi$ curves, and similarly the blue dashed lines represent the constant $\bar \rho$ curves. This figure also confirms that the larger the particle anisotropy, the easier global alignment is achieved.

We remark that these numerical findings are in total agreement with our analytical conclusions. More precisely, at the macroscopic level, this behavior is also evident since the repulsive potential $V_b$ ultimately gives rise to the nematic alignment potential with the factor $ 1-\chi^2(u \cdot u_2)^2 $ in Equation \eqref{eq:W_f}. Consequently, when particles are circular ($\chi =0$), no alignment is observed at the macroscopic level. Also the effect of the potential is less for smaller values of $\chi$. 

We end this section with several snapshots of a particle simulation. In Figure \ref{fig:snapshots}, we start with a completely random state on the upper left figure and run the simulation using the default parameters in Table \ref{tab:model_params}. Above each snapshot, we display the time at which the snapshot was captured with together with the value of the order parameter $\gamma_f$ of the system at that time. The colors represent the nematic direction of the particles. 

\begin{figure}[ht!]
	\centering
	\includegraphics[width=0.32\linewidth]{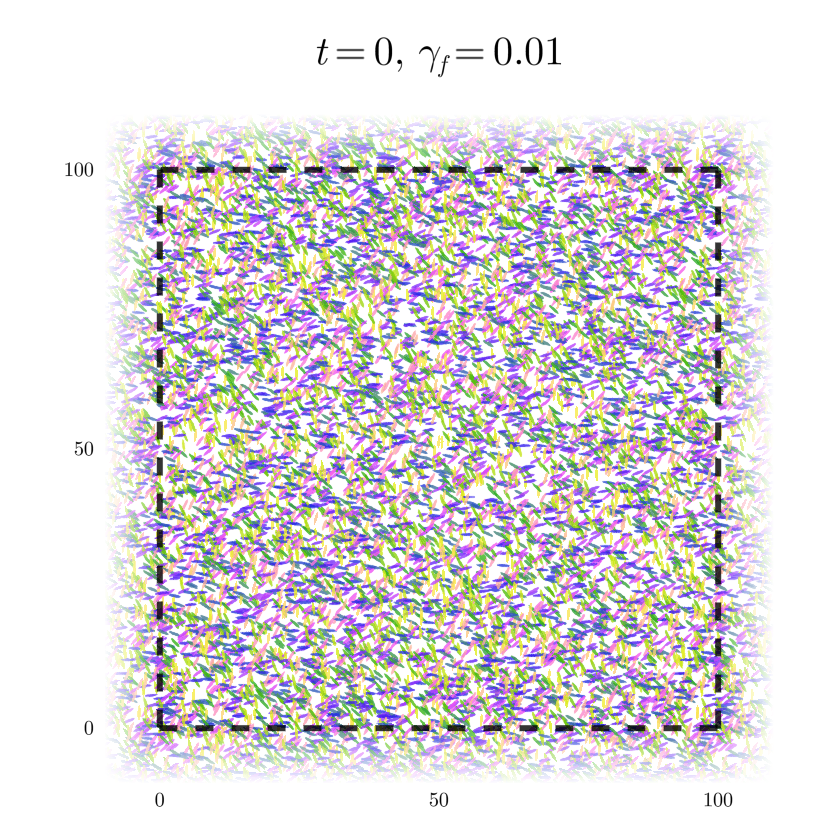}
	\includegraphics[width=0.32\linewidth]{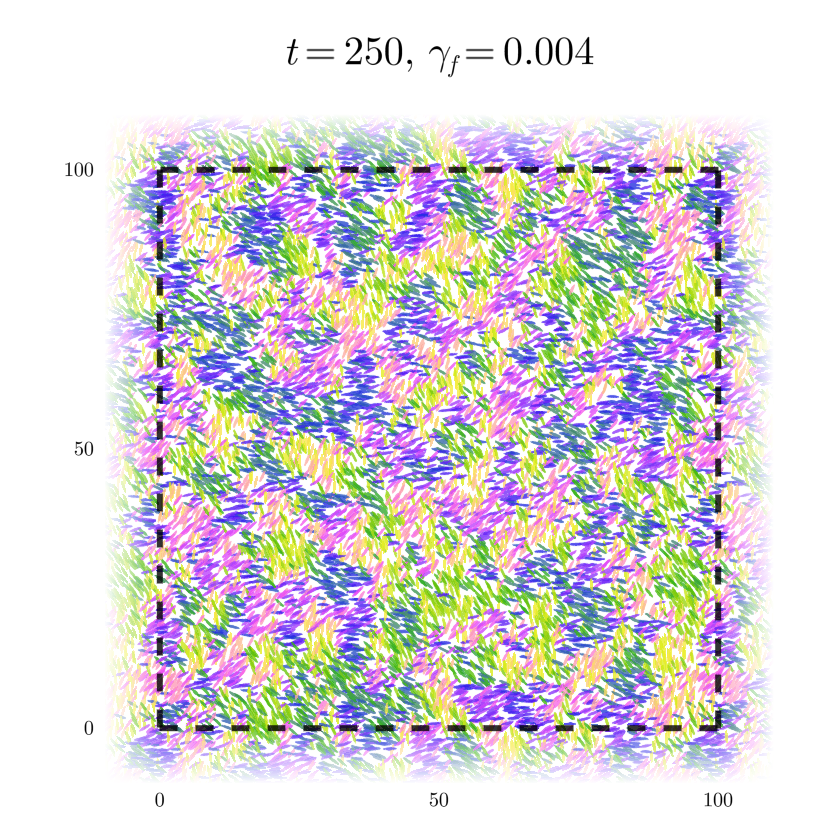}
	\includegraphics[width=0.32\linewidth]{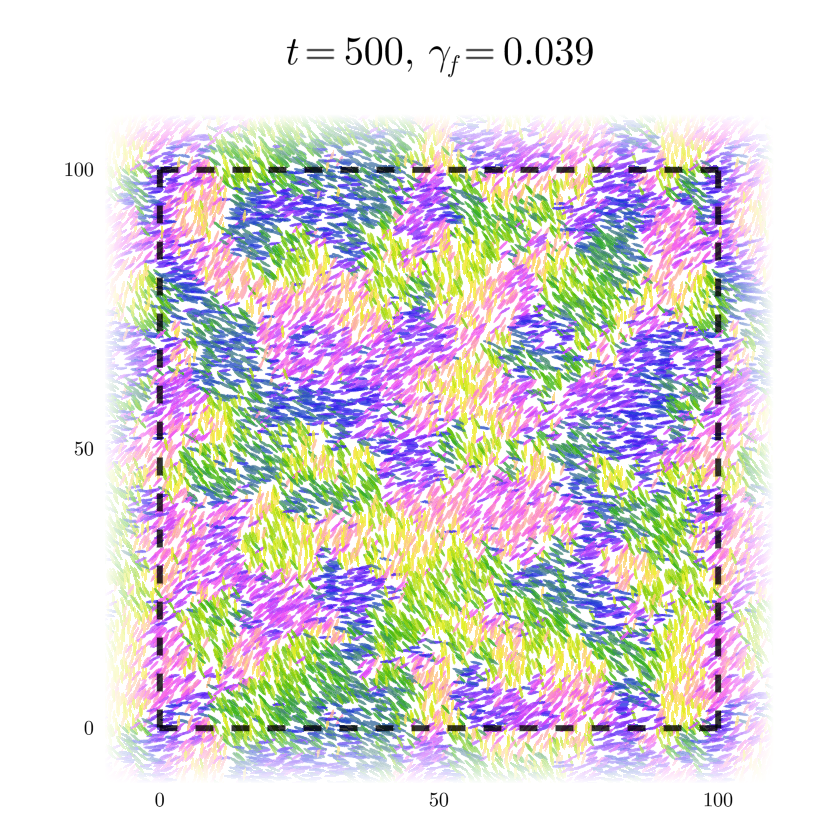}
	\\
	\includegraphics[width=0.32\linewidth]{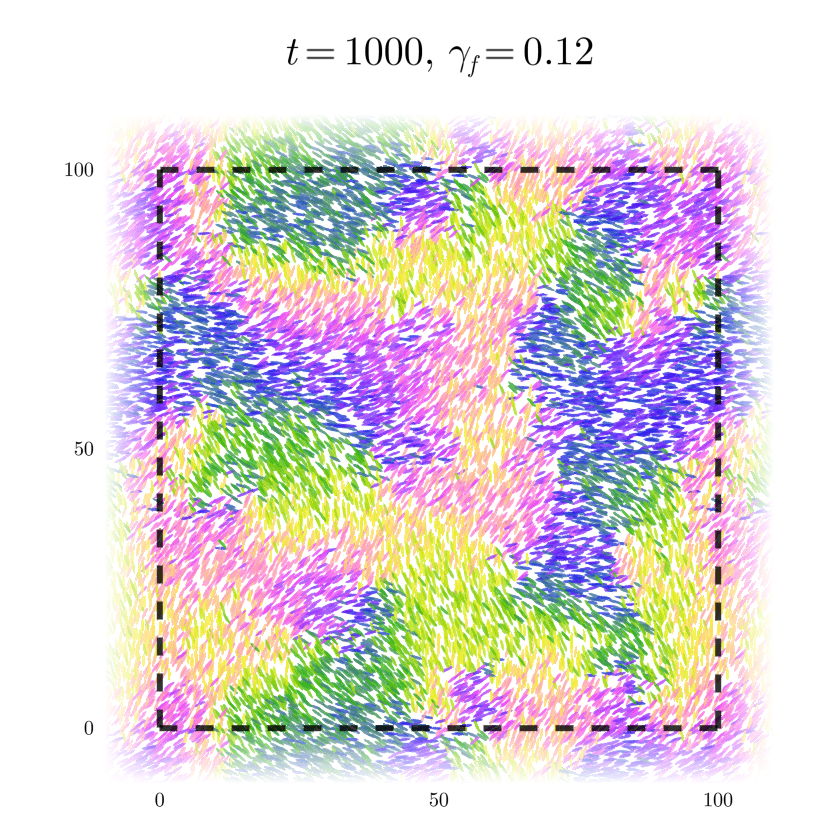}
	\includegraphics[width=0.32\linewidth]{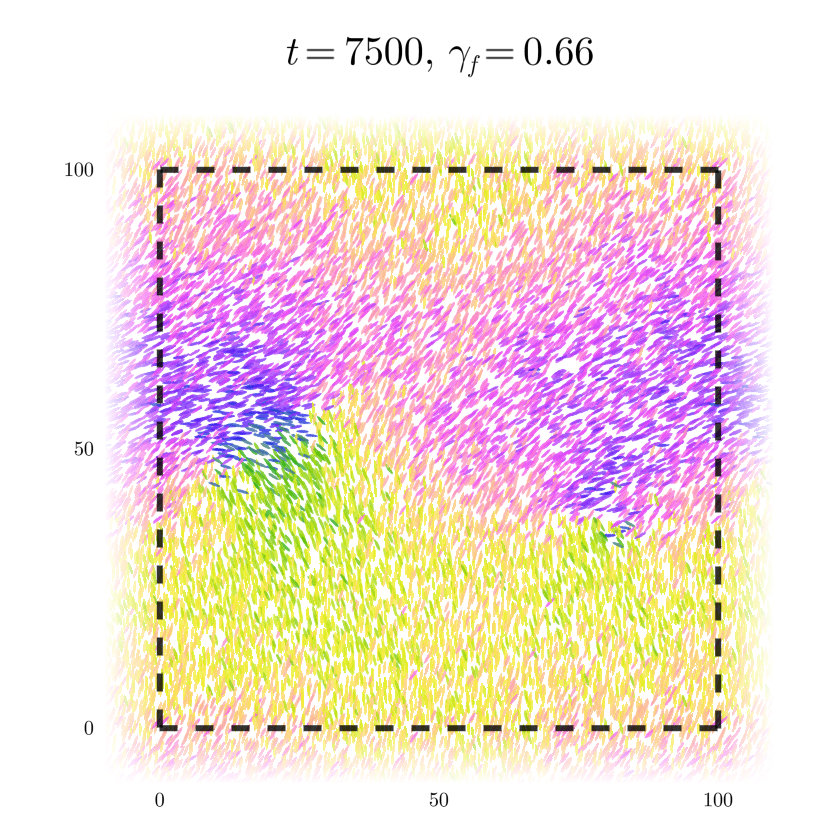}
	\includegraphics[width=0.32\linewidth]{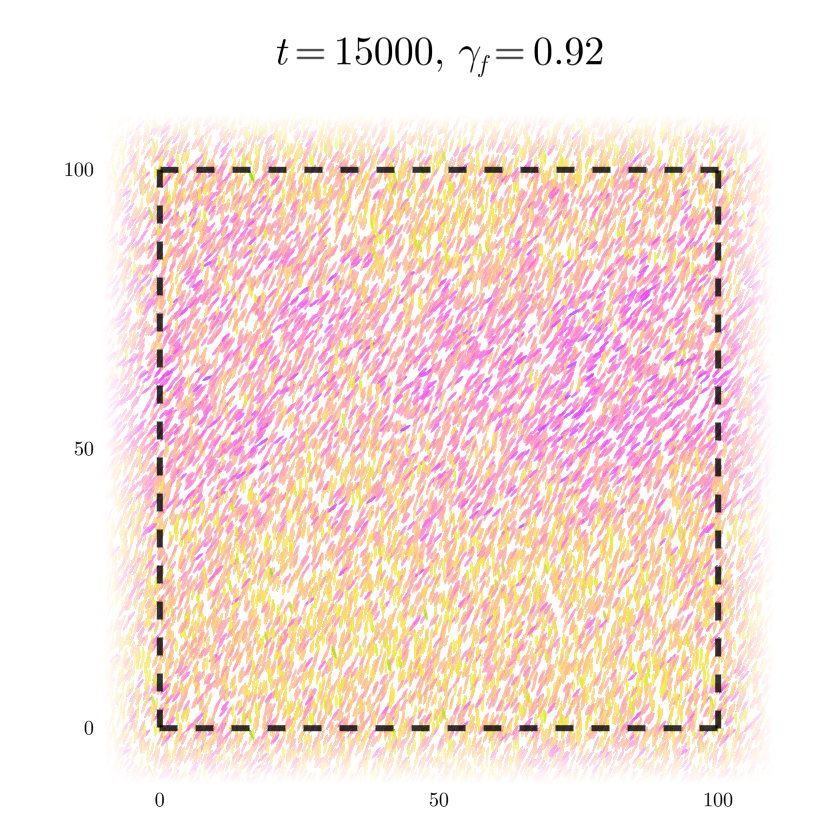}
	 \caption{Snapshots of a particle simulation at different times using the default parameters given in Table \ref{tab:model_params}. Different colors represent the different nematic directions of the particles.
 For the corresponding movie and movies for other cases see link
\url{https://doi.org/10.6084/m9.figshare.27117136.v1}.}
	\label{fig:snapshots}
\end{figure}

The particles start to align over time and hence, the also the order parameter $\gamma_f$ increases in time. Finally, the bottom right snapshot in Figure \ref{fig:snapshots} shows that the system reaches a globally aligned state with $\gamma_f=0.92$, displayed in uniform colors, at the final simulation time $t_{\text{end}}$.

\subsubsection{Numerical comparison of the potentials}
\label{ssec:comparison_potentials}

Another prediction we draw from our mathematical analysis in Section \ref{ssec:potentials} is the lack of alignment in the case of the Berne-Pechukas potential \eqref{eq:potential_BP}. To test this hypothesis numerically, we consider an interpolation between the weighted Gaussian potential and the Berne-Pechukas potential, 
\begin{align}
    V_\xi(u_1,u_2,R) \coloneqq 
    (4\pi)^{-n/2} \det(\Sigma)^{\frac{1}{2} - \xi} \exp(-R^T \Sigma^{-1} R), \label{eq:potential_V_xi}
\end{align} where $\xi \in [0,1]$ and recalling that $\det(\Sigma) = 1 - \chi^2 (u_1 \cdot u_2)^2$. Notice that $V_\xi$ corresponds to 
\begin{itemize}
    \item the weighted Gaussian potential, given by $V_{b_\text{WG}}$ \eqref{eq:weighted_Gaussian_potential}, when $\xi = 0$,
    \item  the non-scaled Gaussian potential $V_{b_\text{NR}}$, obtained by using the scaling factor \eqref{eq:non-rescaled_factor}, when $\xi = \frac{1}{2}$,
    \item the Berne-Pechuckas potential $V_{b_\text{BP}}$, obtained by using the scaling factor \eqref{eq:potential_BP}, when $\xi = 1$.
\end{itemize}

We remark that with increasing $\xi$, the scale of the potential might change, since we have
\begin{align*}
    \frac{V_{\xi}}{V_{b_{\text{WG}}}} = \frac{(1 - \chi^2 (u_1 \cdot u_2)^2)^{\frac{1}{2}-\xi}}{(1 - \chi^2 (u_1 \cdot u_2)^2)^{\frac{1}{2}}} \leq \frac{1}{(1 - \chi^2)^{\xi}}.
\end{align*}

In Section \ref{ssec:sim_results}, we observed (see Figure \ref{fig:vary_lambda_mu}) that increasing $\lambda$ increases the order parameter $\gamma_f$; thus, has a positive effect on the global alignment. Therefore, we consider the scaling $\lambda = ((1-\chi^2)^{2\xi})\lambda'$ so that increasing $\xi$ increases $\lambda'$ and speeds up the alignment. Notice that here we chose the exponent $2\xi$ instead of just $\xi$ to promote even stronger alignment for larger $\xi$ values. For the simulations we use the default parameters listed in Table \ref{tab:model_params}.

\begin{figure}[ht!]
	\centering
    \begin{tabular}[c]{m{8cm} m{7cm}}
	\includegraphics[width=.9\linewidth]{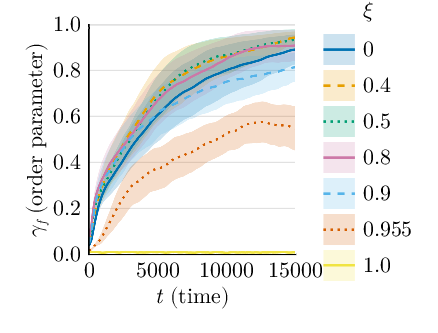}
	& \caption{Trajectories of the order parameter $\gamma_f$ are plotted for varying $\xi$ and the remaining parameters are fixed at their default values given in Table \ref{tab:model_params}. For $\xi=0$, $\xi=\nicefrac{1}{2}$ and $\xi=1$, $V_\xi$ corresponds to the weighted Gaussian potential $V_{b_\text{WG}}$, the non-scaled Gaussian potential $V_{b_\text{NR}}$ and the Berne-Pechuckas potential $V_{b_\text{BP}}$, respectively.}
	\label{fig:sim_vary_potential}
    \end{tabular}
\end{figure}

Figure \ref{fig:sim_vary_potential} displays the trajectories of the order parameter $\gamma_f$ for different values of $\xi$. We clearly observe a sharp transition from higher ordered states to a disordered state (where $\gamma_f$ is nearly $0$) as we increase $\xi=0$ (the weighted Gaussian potential) to $\xi=1$ (the Berne-Pechukas potential). Moreover, we also obtain global alignment for $\xi=\frac{1}{2}$ (the non-scaled Gaussian potential). Thus, our numerical study of the potentials strongly supports our analytical predictions in Section \ref{ssec:potentials}.

\paragraph{Summary of the numerical studies.}

In this section, we conducted comprehensive numerical parameter studies on various parameter pairs: $(D_x, D_u)$, $(\lambda, \mu)$, and $(\chi, \bar{\rho})$. We also compared different interaction potentials by simulating the stochastic particle system \eqref{eq:discrete_system}. Given the complex dynamics governed by \eqref{eq:discrete_system}, it is numerically very challenging to determine all the interesting parameter regimes. Hence, we focused on the transitional ranges where global alignment breaks down. While the observed trends can be intuitively explained at the particle level, they are closely linked to our analytical results in the following ways:
\begin{itemize}
    \item Global alignment fails to emerge for parameters where the macroscopic equation for the mean-nematic alignment $\Omega$ is ill-posed, as characterized by the well-posedness condition \eqref{eq:intro_condition_wellposedness}. This effect is evident in Figures~\ref{fig:heatmap_DxDu} and~\ref{fig:heatmap_lambda_mu}.
    \item Particle shape and density significantly influence the phase transition towards global alignment. As predicted by the macroscopic equations, spherical particles and low densities result in slower rates of alignment or a complete lack thereof.
    \item Our numerical tests confirm that the Berne-Pechukas potential does not lead to global alignment at the particle level, which aligns with the macroscopic perspective presented in Section~\ref{ssec:potentials}.
\end{itemize}

\section{Proof of Theorem \ref{th:macroscopic limit}}
\label{sec:proof}
This section is dedicated to the proof of our main result, Theorem \ref{th:macroscopic limit}. First we give some preliminary tools that will be used later in the proof. 

\subsection{Preliminaries}
\paragraph{Change of variables.} In the sequel, we apply repeatedly the following change of variables $u\mapsto (\theta,v)$ where $u \in \S^{n-1}\setminus\left\{\pm\Omega\right\}$ and $(\theta,v) \in (0,\pi)\times \S^{n-2}$ defined by
\begin{align} \label{eq:change_of_variables}
	u & = (u\cdot\Omega)\Omega +P_{\Omega^\perp}(u)=\cos\theta \Omega + \sin\theta v,\\
	\d u & = \sin ^{n-2} \theta \d \theta \d v,
\end{align} where $\S^{n-2}$ is identified as $\S^{n-1}\cap \Omega^\perp$. For simplicity, we assume that the measure $\d u$ is such that the total mass of the sphere is $1$, i.e., $\int_{\mathbb{S}^{n-1}}\d u=1$. In the course of the proof, we will also apply this change of variables to the Gibbs distribution $\GetaA$ as stated in \eqref{def:GetaA}. Hence, we define
\begin{align} \label{lem:var_trafo_G_h}
	\Getatheta := \frac{e^{\eta(\rho)\cos^2 \theta}}{Z_\eta}, \quad \text{where} \quad  Z_\eta = \int_0^\pi e^{\eta(\rho)\cos^2 \theta'} \d \theta'. 
\end{align} as the transformed Gibbs distribution.
Indeed, notice that the value of the normalizing factor $Z_\eta$ does not depend on $\O$, but it depends on $\eta=\eta(\rho)$. For more details on this change of variables see \cite{DM-A20}.

\paragraph{Derivatives and integrals.} Since we will use various derivatives of $\GetaA$, we collect them below in the following Lemma:

\begin{lem} \label{lem:derivative_GetA}
	Let $\GetaA$ be the Gibbs distribution defined in \eqref{def:GetaA}. 
	Then the following holds, 
	\begin{align}
	\frac{\partial \GetaA}{\GetaA} &=  \eta^\prime(\rho) (\partial \rho)  \lp (u \cdot \Omega)^2 - d_{2,0}(\rho)\rp +2 \eta(\rho) (u\cdot \O) (u\cdot \partial\O), \\ 
	\frac{\nabla_x \GetaA}{\GetaA} &= \eta^\prime(\rho) (\nabla_x \rho) \,  \lp (u \cdot \Omega)^2 - d_{2,0}(\rho)\rp  +2 \eta(\rho) (u\cdot \O) (u\cdot \nabla_x \Omega),   \label{eq:grad_Geta}\\
	\frac{\partial^2\GetaA}{\GetaA} &= \partial^2 \eta \lp (u\cdot\Omega)^2-d_{2,0}(\rho) \rp +\partial \eta \lp 4(u\cdot\Omega)(u\cdot\partial\O)-\partial d_{2,0}(\rho) \rp\\
	&+2\eta (u\cdot\partial\O)^2 +2\eta (u\cdot\Omega) (u\cdot\partial^2\O) + \left (\partial \eta ((u\cdot\O)^2-d_{2,0}(\rho))+2\eta (u\cdot \Omega) (u\cdot \partial\O) \right )^2, %\\ 
	\end{align} 
 where $\partial$ denotes the partial derivative with respect to time or with respect to one of the spatial components, i.e.,  $\p=\p_t$, or $\p=\p_{x_i},\, i=1,\hdots, n$; $\nabla_x \Omega$ is a matrix such that $(\nabla_x \Omega)_{ij}=\partial_{x_j} \Omega_i$ and thus $(u\cdot \nabla_x \Omega)_i = u \cdot \partial_{x_i} \Omega$ and $(\partial^2\Omega)_{i}=\partial^2 \Omega_i$.
\end{lem}
\begin{proof}
	The proof follows directly from straightforward computations. 
\end{proof}

\begin{lem} \label{lem:model_integrals}
	Let $h_\eta$ be any odd function, $w \in \R^n$, and $k \in \mathbb{N} \cup 0$. Then the following holds for $\GCI := \heo\,  \proopp u$: 
\begin{align}
	 & \ints (u\cdot \Omega)^k \, \GetaA \GCI\, \d u =0,   \label{eq:intG}\\
	& \ints (u\cdot w)(u\cdot \Omega)^k \, \GetaA \GCI \, \d u =
	\frac{c_{k,2}}{n-1}P_{\Omega^\perp} w ,  \label{eq:intGw}\\
	& \ints (u\cdot \Omega)^k \, \partial\GetaA \GCI \, \d u =
	\eta(\rho)\frac{2c_{k+1,2}}{n-1}\partial \O, \label{eq:intGpartial}\\
	& \ints (u\cdot w) (u\cdot \Omega)^k \, \partial\GetaA \GCI \, \d u =
	\eta'(\rho) (\partial \rho) (c_{k+2,2}-d_{2,0} c_{k,2})\frac{1}{n-1} w \quad \mbox{ for }w\perp \Omega, \label{eq:intGw_partial}\\
	& \ints (u\cdot\Omega)^{2k} (u\cdot w)^2  \, \GetaA \GCI \, \d u =0 \quad \mbox{ for }w\perp \Omega,  \label{eq:Gw2_squared}\\
	&\ints \partial^2 \GetaA \GCI \d u = \frac{4}{n-1} \lp c_{1,2}+\eta(\rho) (c_{3,2}-d_{2,0}c_{1,2})\rp \partial\eta\partial\Omega + \frac{2c_{1,2}}{n-1}\eta(\rho) \proopp\partial^2 \O, \label{eq:intLaplacian}
\end{align} where the partial derivatives are with respect to time or space, i.e., $\partial=\partial_t, \partial_{x_i}$ with $x_i$ being the $i$-th component of $x$ and the coefficients $c_{k,p}$, $d_{k,p}$ are given in \eqref{eq:def_ckl} and \eqref{eq:def_dkl}. Moreover, it also holds that
\begin{align*}
	d_{k,p}=0 \quad \mbox{ for $k$ odd}\qquad \mbox{and} \qquad c_{k,p}=0 \quad \mbox{ for $k$ even}.
\end{align*}
\end{lem}
\begin{proof}
	The proof of this lemma can be found in Appendix \ref{appendix_lem_model_integrals}.
\end{proof}

Finally, notice that 
\begin{align}
	\partial \Omega \perp \Omega \label{eq:perpedincular_derivativeOmega}
\end{align}
for $\partial=\partial_t, \partial_{x_i}$ (since $\O \cdot \partial \O = \partial |\O|^2/2=0$). Therefore, it holds that
\begin{align}\label{eq:orthogonalityOmega}
	P_{\Omega^\perp} \partial\O = \partial \O.
\end{align}

Next, we consider the following integrals with respect to $v \in \S^{n-2}$.
\begin{lem}[Lemma 4.1 in \cite{DM-A20}] \label{lem:v_integrals}
	Let $n\geq 2$ and $v \in \S^{n-2}$. Then, we have 
	\begin{align}
			\int_{\S^{n-2}} v^{\otimes (2k+1)} \d v &=  0, \quad \forall k \in \mathbb{N}, \label{eq:v_odd}\\
			\int_{\S^{n-2}} v\otimes v \d v &= \frac{1}{n-1} \proopp, \label{eq:v_squared}\\
			\int_{\S^{n-2}} v\otimes v \otimes v \otimes v \d v & =  \frac{1}{(n-1)(n+1)}\Gamma \label{eq:v_4tensor},
	\end{align} where $\Gamma$ is given in \eqref{eq:Gamma}. 
\end{lem}
The proof of Lemma \ref{lem:v_integrals} can be found in Lemma 4.1 of \cite{DM-A20}. Thus, we skip it here. 
\medskip 

Now, with these preliminary results, we first present the limit of $f^\eps$ as $\eps \to 0$.

\subsection{Limit for \texorpdfstring{$f^\eps$}{f\^e}}

\begin{lem}[Limit of $f^\eps$ and $\rho_{f^\eps} Q_{f^\eps}$] \label{lem:limitf_and_Q} Let $f^\eps$ be the solution of Equation \eqref{eq:dimless_kinetic_eq}. Under Assumption \ref{as:A}, we have that
	\begin{align} \label{eq:limit_f}
		f^\eps (t,x,u) \to \rho(t,x) \GetaA(t,x) \quad \mbox{as } \eps \to 0, 
	\end{align}
	for any $(t,x)$ such that $\rho(t,x) > \rho_*$, for $\rho_*$ as given in Proposition \ref{prop:def_eta}. Moreover, in this same regime of $(t,x)$, it also holds that
	\begin{align} \label{eq:limit_rhoQ_2}
		\rho_{f^\eps} Q_{f^\eps} \to \frac{\eta(\rho)}{\alpha}A_\O, \quad \mbox{as } \eps \to 0, 
	\end{align} where $\alpha$ is given in \eqref{eq:def:alpha}, $\eta=\eta(\rho)$ is given in \eqref{eq:etaroots}, and $A_\O$ is given in \eqref{def:A_Om}.
\end{lem}

\begin{proof}
	By Assumption \ref{as:A}, $f^\eps$ converges to some function $f^0$ as $\eps \to 0$. At the same time, from Equation \eqref{eq:rescaled_kinetic} we have that $C(f^\eps)=\mathcal{O}(\eps^a)$ with $a \in (0,2]$ and thus, taking the limit, $C(f^0)=0$ as $\eps \to 0$. Therefore, the limit $f^0$ must belong to the kernel of the operator $C$, which was characterized in Lemma \ref{lem:stable_equilibria}. Therefore, we conclude that $f^0=\rho(t,x)\GetaA(t,x)$ for any $(t,x)$ such that $\rho(t,x) > \rho_*$ and for some $\Omega \in \S^{n-1}$. Finally, the limit \eqref{eq:limit_rhoQ_2} is a direct consequence of \eqref{eq:limit_rhoQ} in Lemma \ref{lem:stable_equilibria}.
\end{proof}

\subsection{Derivation of Equation (\ref{eq:macro_rho})} 
Integrating the kinetic equation \eqref{eq:rescaled_kinetic} with respect to $u$, and using Lemma \ref{lem:limitf_and_Q}  in the limit $\eps \to 0$, we obtain
\begin{align}\label{eq:rho_proof}
	\p_t \rho + \frac{\mu \chi^2}{\alpha} \nabla_x \cdot \lp \rho \int_{\S^{n-1}} \left( \nabla_x \lp u\cdot ( \eta(\rho)A_\O ) u \rp  \GetaA \right) \d u \rp+ \mu\lp\frac{\chi^2}{n}-1\rp \nabla_x \cdot \lp \rho 
	\nabla_x\rho \rp - D_x \Delta_x \rho = 0. 
\end{align} Notice that, above, we used the facts that the term in $B_f$ vanishes and $\ints C(f^\eps) du =0$ for all $\eps >0$, both thanks to the divergence theorem. 

Next, we compute the integral in the second term of \eqref{eq:rho_proof},
\begin{align*}
	\int_{\S^{n-1}} \nabla_x &\lp \etarho u\cdot A_\O u \rp   \GetaA  \d u   \\
	& =  \int_{\S^{n-1}}    \nabla_x \lp \etarho \left[(u\cdot \O)^2-\frac{1}{n}\right] \rp  \GetaA  \d u  \\
	& =  \ints \nabla_x \eta(\rho) \left( (u\cdot \O)^2-\frac{1}{n}\right) \GetaA \d u + \ints 2 \eta(\rho) (u \cdot \O) (u\cdot \nabla_x \O) \GetaA \d u \\
	& = \frac{n-1}{n}S_2(\eta)\, \nabla_x \etarho,
\end{align*} where in the last equality we used the fact that the second integral is zero.  This can be computed similarly as for the integral in Equation \eqref{eq:intGw} for $k=1$ and $w=\partial \Omega$ with $\partial \Omega \perp \Omega$ (remembering that the matrix $\nabla_x\Omega$ is such that component-wise it corresponds to $(\nabla_x\Omega)_{ij}=\partial_{x_j} \Omega_i$).  

Finally, substituting this last expression into \eqref{eq:rho_proof} and using that $\nabla_x\eta(\rho)=\eta'(\rho) \nabla_x \rho$, we can combine the terms and obtain Equation \eqref{eq:macro_rho} for the mass density $\rho$. 

\subsection{Derivation of Equation (\ref{eq:macro_Omega}).} Classically, to obtain macroscopic equations from kinetic equations, one needs to find functions $\psi=\psi(u)$ that are collision invariant for the operator $C$, i.e., functions $\psi$ such that, for all $f$, the following holds:
\begin{align*}
	\int_{\S^{n-1}}C(f)\psi(u)\d u=0.
\end{align*} 
The collision invariants correspond to conserved quantities. Typically, the number of macroscopic equations we obtain is the same as the number of collision invariants. For example, in our case, if $\psi$ is a constant, then $\psi$ is a collision invariant, which is related to the conservation of mass.
Therefore, to obtain an equation of the mean-nematic direction $\O$, we would like to find its associated collision invariant(s). However, $\O$ does not have any associated conserved quantity, i.e., one can check that the only collision invariants for our operator $C$, given in \eqref{eq:defC}, are the constants \cite{DFL22}.

To overcome this difficulty, the concept of Generalized Collision Invariant (GCI) was introduced in \cite{DM08} for a model of collective dynamics where the momentum is not preserved. This idea, then, has been applied to other models of collective dynamics, see, e.g., \cite{DFM-AT18,DM-A20}. Moreover, the GCI for the operator $C$ was already computed in \cite{DFL22}. Next, we state the main result regarding the GCI that we will be using in the sequel. The original source \cite{DFL22} contains a full proof of it. 

\begin{prp}[Generalized Collision Invariant defined in \cite{DFL22} and Proposition 9 in \cite{DFL22}] \label{prop:GCI_property}
		Let \linebreak $f \colon \S^{n-1} \to \R$ be twice continuously differentiable such that $Q_f\neq 0$ and $Q_f$ is the $Q$-tensor. Since $Q_f$ is a symmetric matrix, all its eigenvalues are real. Suppose that the largest eigenvalue of $Q_f$ is unique and denote $\Omega_f$ as the associated normalized eigenvector (which is unique up to sign). Then it holds that
		\begin{align*}
			\int_{\S^{n-1}} C(f)\,  \boldsymbol{\psi}_{\eta_f,\Omega_f} \d u =0, \quad \mbox{ for }\quad \boldsymbol{\psi}_{\eta_f,\Omega_f} :=h_{\eta_f}(u\cdot\Omega_f) P_{\Omega^\perp_f}u,
		\end{align*} where $h_\eta$ is the unique solution of
\begin{align*}
			-(1-r^2)^{\frac{n-1}{2}} e^{\eta r^2}(2\eta\, r^2 +n-1) h_\eta +\frac{\d}{\d r}\left[(1-r^2)^{\frac{n+1}{2}}e^{\eta r^2} \frac{\d h}{\d r}\right]=r(1-r^2)^{\frac{n-1}{2}}e^{\eta r^2},
		\end{align*}
  in the functional space
  \begin{align}
	\mathcal{H}=\left\{ h: (-1,1)\to \R \, \, \Big  | \,  \int_{-1}^1 (1-r^2)^{\frac{n-1}{2}} |h(r)|^2 \d r < \infty, \,  \int_{-1}^1 (1-r^2)^{\frac{n+1}{2}} \left |  h'(r) \right |^2 \d r < \infty \right\}.
\end{align}
  Moreover, $h_\eta$ is an odd function with $h_\eta(r)\leq 0$ for $r\geq 0$. 
\end{prp}  
  
A characterization of the GCI is given by the following definition and can be found in \cite{DFL22}:

\begin{dfn}\label{def:vec_GCI}
		Let $(\eta, \Lambda) \in (0,\infty) \times \mathcal{U}_n^0$. Then, we define the function $\psi_{\eta \Lambda} \colon \S^{n-1}~\to~\R^n$ as the unique solution (in $H^1(\S^{n-1})=\left\{ k \in H(\S^{n-1})| \int_{\S^{n-1}} k(u) \d u =0 \right\}$) of the following equation:
  \begin{align*}
      \nabla_u \cdot (G_{\eta \Lambda}(u) \nabla_u \psi_{\eta \Lambda})=(u\cdot \O_\Lambda) P_{\O_\Lambda^\perp} u \, G_{\eta \Lambda}(u) \qquad \forall u \in \S^{n-1},
  \end{align*} where  $\mathcal{U}^0_n$ is the set of symmetric trace-free $n\! \times\! n$-matrices whose principal eigenvalue is equal to $\frac{n-1}{n}$ and is simple.
\end{dfn}

Thanks to the GCI for the collision operator $C$ we can now derive an equation for the mean-nematic direction $\Omega$. Therefore, we multiply the kinetic equation given in \eqref{eq:dimless_kinetic_eq} by $\psi_{\eta \O_f}$ and integrate with respect to $u$. Then, the term in $C$ vanishes by the previous proposition. Therefore, we state the following result:

\begin{prp}
	Under the assumptions of Theorem \ref{th:macroscopic limit}, it holds that
	\begin{align} 	\label{eq:macro_GCI_Omega}
			\int_{\S^{n-1}}   \mathcal{F}_u(\rho G_{\eta(\rho) A_\Omega} ) \GCI \, \d u =0, \quad \mbox{ for } \quad \GCI = h_{\eta(\rho)}(u\cdot \Omega) P_{\Omega^\perp}u,
	\end{align} where $\eta=\eta(\rho)$ is given by \eqref{eq:etaroots}, $h_\eta=h_{\eta(\rho))}(u\cdot \Omega)$ is defined in Proposition \ref{prop:GCI_property} and $\mathcal{F}_u(f)$ is defined as  
	\begin{multline*}
		\mathcal{F}_u (f) =\partial_t f  + \mu \chi^2  \nabla_x \cdot (\nabla_x (u^T \rho_f Q_f u) f) \\ + \mu \left ( \frac{\chi^2}{n} -1 \right ) \nabla_x \cdot  ((\nabla_x \rho_f) f) - \lambda \1_{a=2} \nabla_u \cdot ((\nabla_u B_f) f) -D_x  \Delta_x f.
	\end{multline*} 
\end{prp}
\begin{proof}
	The proof is a direct consequence of Assumption \ref{as:A}, Lemma \ref{lem:limitf_and_Q} and Proposition \ref{prop:GCI_property}.
\end{proof}

\subsubsection{Limit of the terms not involving \texorpdfstring{$B_f$}{B\_f}}

Now, we have all the ingredients to derive the macroscopic equation for $\Omega$ in a straightforward way. 
We apply Lemma \ref{lem:limitf_and_Q} in the limit as $\eps \to 0$ and rewrite Equation \eqref{eq:macro_GCI_Omega} as  
\begin{align}
 \label{eq:Is}
	I_1+\mu\sigma I_2+\mu\lp\frac{\chi^2}{n}-1 \rp I_3-\lambda \1_{a=2}I_4 -D_x I_5=0,
\end{align} where 
\begin{align*}
	I_1&= \int_{\S^{n-1}} \partial_t \lp \rho \GetaA \rp \GCI \d u, \\    
	I_2&= \int_{\S^{n-1}} \nabla_x \cdot  \lp \nabla_x \lp u\cdot \eta(\rho) A_\O u  \rp \rho \GetaA\rp \GCI \d u,\\ 
	I_3&= \int_{\S^{n-1}} \nabla_x \cdot  \lp (\nabla_x\rho) \, \rho \GetaA\rp \GCI \d u , \\ 
	I_4&=\int_{\S^{n-1}} \nabla_u \cdot \lp  \nabla_u  \lp B_{\rho \GetaA}\rp \rho \GetaA \rp \GCI \d u, \\ 
	I_5 &=\int_{\S^{n-1}} \Delta_x \lp \rho \GetaA \rp \GCI \d u.
\end{align*}
Notice that while writing $I_2$ we used \eqref{eq:limit_rhoQ}.
In this section, we will focus only on the terms $I_1,I_2,I_3$ and $I_5$ since the term $I_4$ appears exclusively if we choose $a=2$. We will compute $I_4$ in the next section.

\begin{lem} \label{lem:limit_I1_I2_I4}
	The following equalities hold:
	 \begin{align*}
		I_1= & \frac{2 \eta(\rho) c_{1,2}}{n-1} \ \rho\p_t \O, \\
		I_2=& \frac{2\eta(\rho) c_{1,2}}{n-1}\lp 2\frac{\eta^\prime(\rho)}{\eta(\rho)} \rho  +1+ \rho  \eta^\prime(\rho) \lp 2\frac{c_{3,2}}{c_{1,2}}-\frac{1}{n}-d_{2,0} \rp\rp (\nabla_x \rho\cdot\nabla_x)\O + \frac{2 \eta(\rho) c_{1,2}}{n-1}\ \rho \proopp \Delta_x \O,\\
		I_3 =& \frac{2 \eta(\rho) c_{1,2}}{n-1}\ \rho(\nabla_x\rho\cdot\nabla_x)\O,\\
		I_5=&\frac{2 \eta(\rho) c_{1,2}}{n-1} \ \rho P_{\Omega^\perp}\Delta_x\O+ \frac{2 \eta(\rho) c_{1,2}}{n-1}\ 2\lp \rho \eta'\left(\frac{1}{\eta(\rho)}+\frac{c_{3,2}}{c_{1,2}}-d_{2,0})\right )+1\rp(\nabla_x\rho\cdot\nabla_x)\O.
	\end{align*}
\end{lem}
\begin{proof}
	 We start with computing $I_1$,
	 \begin{align*}
	 	I_1 &= \ints \p_t (\rho \GetaA) \GCI \d u 
	 	=  \p_t \rho \int_{\S^{n-1}} \GetaA \GCI \d u + \rho \int_{\S^{n-1}} \p_t \GetaA \GCI \d u.
	 \end{align*} 
 The first integral above on the last part of the equality is of the form \eqref{eq:intG} for $k=0$, so it is zero. The second integral is of the form \eqref{eq:intGpartial} for $k=0$ and so $I_1$ becomes
 \begin{align*}
 	I_1= \frac{2 \eta(\rho) c_{1,2} (\rho) }{n-1}\ \rho\p_t \O.
 \end{align*}
Next, we compute $I_2$, 
\begin{align*}
	I_2&= \int_{\S^{n-1}} \nabla_x \cdot  \lp \nabla_x \lp  u\cdot \eta(\rho) A_\O u  \rp \rho \GetaA\rp \GCI \d u\\
	& = \int_{\S^{n-1}} \Delta_x\lp  u\cdot \eta(\rho) A_\O u   \rp \rho \GetaA \GCI \d u +\int_{\S^{n-1}} \nabla_x\lp  u\cdot  \eta(\rho) A_\O u  \rp \cdot \nabla_x \rho \ \GetaA \GCI \d u  \\
& \quad 	+\rho \int_{\S^{n-1}} \nabla_x\lp  u\cdot  \eta(\rho) A_\O u \rp \cdot \nabla_x \GetaA \GCI \d u  =: I_2^1+I_2^2+I_2^3.
\end{align*} Proceeding further, we compute each $I^i_2$, $i\in \{1,2,3\}$ separately.  Using  \eqref{def:A_Om}, we have 
\begin{align*}
	u\cdot \eta(\rho) A_\O u = \eta(\rho) \lp(u\cdot \O)^2 - \frac{1}{n} \rp.
\end{align*}
The Laplacian of this expression reads
\begin{align*}
	\Delta_x (u\cdot\eta(\rho) A_\O u) = &\Delta_x \eta(\rho) \lp (u\cdot \O)^2 -\frac{1}{n} \rp + 4\nabla_x \eta(\rho) \cdot  (u\cdot \O) (u \cdot \nabla_x \O)   \\
	&\quad +2 \eta(\rho) (u\cdot \nabla_x \O)^2 +2 \eta(\rho) (u \cdot \O) (u \cdot \Delta_x \O). 
\end{align*}
Subsequently, using this, we obtain
\begin{align*}
	I_2^1 = &\int_{\S^{n-1}} \Delta_x\lp  u\cdot \eta(\rho) A_\O u  \rp \rho \GetaA \GCI \d u \\
	= &\, \rho  \Delta_x \eta(\rho) \int_{\S^{n-1}}  \lp (u\cdot \O)^2 -\frac{1}{n} \rp \GetaA \GCI \d u  
 \\ &\, \, + 4\eta^\prime(\rho) \rho\, \sum_{i=1}^n\partial_{x_i} \rho \lp \int_{\S^{n-1}}   (u\cdot \O) (u \cdot \partial_{x_i} \O) \GetaA \GCI \d u \rp\\
	&\, \, + 2 \eta(\rho) \rho \int_{\S^{n-1}} (u\cdot \nabla_x \O)^2 \GetaA \GCI \d u  + 2 \eta(\rho)  \rho \int_{\S^{n-1}} (u \cdot \O) (u \cdot \Delta_x \O) \GetaA \GCI \d u. 
\end{align*} We further simplify each integral above. The first and the third integrals are equal to zero as they are of the forms \eqref{eq:intG} for $k=2$ and $k=0$ and \eqref{eq:Gw2_squared} for $k=2$, respectively. The second and the forth integrals are of the form \eqref{eq:intGw} with $k=1$ (and recall \eqref{eq:orthogonalityOmega}). Combining these we obtain
\begin{align*}
	I_2^1= 4\eta^\prime(\rho) \rho \frac{c_{1,2}}{n-1}  (\nabla_x \rho\cdot\nabla_x)\O + 2\eta(\rho) \rho \frac{c_{1,2}}{n-1} \proopp \Delta_x \O.
\end{align*}
Next, for $I_2^2$ we have 
\begin{align*}
	I_2^2 &=\nabla_x \rho\cdot\int_{\S^{n-1}} \nabla_x\lp  u\cdot \eta(\rho) A_\O u  \rp  \ \GetaA \GCI \d u \\
	&=|\nabla_x\rho|^2 \eta^\prime(\rho)\int_{\S^{n-1}} \lp (u\cdot \O)^2-\frac{1}{n} \rp  \ \GetaA \GCI \d u 
 \\&\qquad \qquad \qquad +2\eta(\rho) \sum_{i=1}^n \partial_{x_i}\rho  \lp \int_{\S^{n-1}}  (u\cdot \O)(u\cdot \partial_{x_i} \O)   \ \GetaA \GCI \d u \rp.
\end{align*}
Noticing that the first integral above is of the form \eqref{eq:intG} with $k=2$ and $k=0$ and the second integral is of the form \eqref{eq:intGw} with $k=1$ and recalling \eqref{eq:orthogonalityOmega}, we conclude
\begin{align*}
	I_2^2 = 2\eta(\rho) \frac{c_{1,2}}{n-1} \lp\nabla_x \rho \cdot \nabla_x\rp \O.
\end{align*}
 Similarly, we obtain (here we additionally use \eqref{eq:grad_Geta} for $\nabla_x \GetaA$)
 \begin{align*}
 	I_2^3 &=\rho \int_{\S^{n-1}} \nabla_x\lp u\cdot \eta(\rho) A_\O u  \rp \cdot \nabla_x \GetaA \GCI \d u \\
 	 &= \rho \eta^\prime(\rho) \sum_{i=1}^n\partial_{x_i}\rho \int_{\S^{n-1}} \lp  (u\cdot \O)^2-\frac{1}{n} \rp  \partial_{x_i} \GetaA \GCI \d u 
   \\ &\qquad \qquad \qquad +2 \eta(\rho) \rho \int_{\S^{n-1}}  (u\cdot \O) \lp (u\cdot \nabla_x \O) \cdot \nabla_x \GetaA \rp \GCI \d u.
 \end{align*}
The first integral above is of the form \eqref{eq:intGpartial} for $k=2$ and $k=0$ (for the terms $(u\cdot \Omega)^2$ and $1/n$), respectively; and the second integral is of the form \eqref{eq:intGw_partial} for $k=1$, therefore, we have that
\begin{align*}
	I_2^3=2\rho \eta(\rho) \eta^\prime(\rho) \lp \frac{2c_{3,2}-c_{1,2}/n-d_{2,0}c_{1,2}}{n-1} \rp \lp \nabla_x\rho \cdot \nabla_x\rp \Omega .
\end{align*} 
Combining these we compute $I_3$ analogously to obtain
\begin{align*}
	I_3&= \int_{\S^{n-1}} \nabla_x \cdot  \lp (\nabla_x\rho) \, \rho \GetaA\rp \GCI \d u=  2\rho\eta(\rho)\frac{c_{1,2}}{n-1}(\nabla_x\rho\cdot\nabla_x)\O.
\end{align*}
Finally, for $I_5$ we have 
\begin{align*}
	I_5=&\int_{\S^{n-1}} \Delta_x \lp \rho \GetaA \rp \GCI \d u\\
	= & \Delta_x \rho \int_{\S^{n-1}}  \GetaA  \GCI \d u + 2 \sum_{j=1}^n\partial_{x_d}\rho \int_{\S^{n-1}}  \partial_{x_d}\GetaA  \GCI \d u + \rho \int_{\S^{n-1}} \Delta_x \GetaA  \GCI \d u. 
\end{align*}
The first integral above is of the form \eqref{eq:intG} for $k=0$, so it is zero; the second integral is of the form \eqref{eq:intGpartial} for $k=0$; for the third integral we use \eqref{eq:intLaplacian} and that $\partial \eta =\eta'(\rho) \partial(\rho)$. Hence, we conclude
\begin{align*}
	I_5 = \frac{2 c_{1,2}}{n-1}\eta(\rho) \, \rho P_{\Omega^\perp}\Delta_x\O+ \frac{4}{n-1}\lp \rho \eta'[c_{1,2}+\eta(c_{3,2}-d_{2,0}c_{1,2})]+\eta(\rho)c_{1,2}\rp(\nabla_x\rho\cdot\nabla_x)\O. 
\end{align*}
This completes the proof. 
\end{proof}

\subsubsection{Limit of the term with \texorpdfstring{$B_f$}{B\_f}}
\label{sec:termBf}

In this section, we compute the remaining integral $I_4$ in Equation \eqref{eq:Is}, which includes the term $B_f$. First, we prove the following lemma:

\begin{lem} \label{lem:Bf_GCI}
	Let $B_f$ be given by \eqref{eq:B_f} with $f=\rho \GetaA$.  Then the following holds
	\begin{align*}
	\int_{\S^{n-1}} \nabla_u\cdot (\nabla_u B_{\rho \GetaA}\ \rho G_{\eta A_\Omega}) \ h_\eta(u\cdot \Omega) P_{\Omega^\perp} u \d u =\rho \int_{\S^{n-1}} B_{\rho \GetaA}(u) \, (\udoto) \proopp u \ \ga \d u.
	\end{align*}
\end{lem}

\begin{proof} Applying integration by parts twice, we rewrite the integral as 
	\begin{align*}
		\ints \nabla_u\cdot (\nabla_u B_{\rho \GetaA} & \rho G_{\eta A_\Omega})  h_\eta(u\cdot \Omega) P_{\Omega^\perp} u \d u 
		\\&=  - \ints (\nabla_u B_{\rho \GetaA}\ \rho G_{\eta A_\Omega})  \cdot  \nabla_u \left [ h_\eta(u\cdot \Omega) P_{\Omega^\perp} u \right ] \d u\\
		&= \rho\ints  B_{\rho \GetaA} \nabla_u \cdot \left ( G_{\eta A_\Omega}  \nabla_u[h_\eta(u\cdot \Omega) P_{\Omega^\perp} u] \right ) \d u.
	\end{align*} We conclude the result since the GCI $\psi =h_\eta(u\cdot \Omega) P_{\Omega^\perp} u $ satisfies
\begin{align*}
	\nabla_u \cdot (\ga \nabla_u \psi) = (\udoto)\ \proopp u \, \ga
\end{align*} by Definition \ref{def:vec_GCI}.
\end{proof}
Now, we go back to Equation \eqref{eq:Is} to rewrite $I_4$ using Lemma \ref{lem:Bf_GCI},
\begin{align} 	\label{eq:B_alternative}
	I_4 =\rho \int_{\S^{n-1}} B_{\rho \GetaA}(u) 	(\udoto) \proopp u  \ga \d u.
\end{align}
First,  we recast $B_{\rho \GetaA}=[B]+[B]_{\text{even}}$, as stated in Equation \eqref{eq:B_f}, by expanding
\begin{align*}
\Sigma(u, u_2)= (\ell^2-d^2)(u\otimes u) + (\ell^2-d^2) (u_2\otimes u_2) + 2 d^2 \Id 
\end{align*} and the factor $(1- \chi^2(u\cdot u_2)^2)$. Hence, we obtain 
\begin{align*}
	[B] &:= \frac{(\ell^2-d^2)}{4} ((u\otimes u) : D_x^2) \rho - \frac{\chi^2}{4}(\ell^2-d^2) \int_{\S^{n-1}} (u\cdot u_2)^2  ((u\otimes u): D_x^2)(\rho \GetaA (u_2)) \d u_2\\
	&\quad - \frac{\chi^2}{4}(\ell^2-d^2)\int_{\S^{n-1}} (u\cdot u_2)^2  \ ((u_2\otimes u_2): D_x^2)(\rho \GetaA (u_2)) \d u_2 
 \\ &\quad -\frac{\chi^2}{2}d^2\int_{\S^{n-1}} (u\cdot u_2)^2 \Delta_x(\rho \GetaA(u_2)) \d u_2.
\end{align*}
and 
\begin{align*}
	 [B]_{\text{even}} := \frac{(\ell^2-d^2)}{4} \int_{\S^{n-1}} \lp (u_2\otimes u_2): D^2_x\rp (\rho \GetaA (u_2)) \d u_2 +\frac{d^2}{2}\Delta_x \rho.
\end{align*}
Now, using the change of variables \eqref{eq:change_of_variables} in \eqref{eq:B_alternative}, we have that 
\begin{align*}
	I_4 =\rho\int_{\S^{n-2}}  v \int_0^\pi B_{\rho \GetaA}(u(\theta)) \, \cos\theta \sin\theta  \ \Getatheta\ \sin^{d-2}\theta  \d\theta \d v.
\end{align*}
Notice that the terms in $[B]_{\text{even}}$ are independent of $u$, so they have zero contribution to $I_4$ (since they give an integrand odd in $v$).
Now, we rewrite $I_4$ using $[B]$ and the tensor product:
\begin{multline*}
	 I_4 =\rho \frac{(\ell^2-d^2)}{4} [(H^r_2 : D^2_x)]_{[2,3:1,2]}\rho  - \rho \frac{\chi^2}{4}(\ell^2-d^2) [H^r_4: (D^2_x \otimes (\rho H_2)]_{[2,3,4,5:1,2,3,4]}  \\
	- \rho \frac{\chi^2}{4}(\ell^2-d^2) [H^r_2\otimes D^2_x : (\rho H_4)]_{[2,3,4,5:1,2,3,4]} - \rho \frac{\chi^2}{2}d^2 [H^r_2 : \Delta_x (\rho H_2)]_{[2,3:1,2]}, 
\end{multline*} where in the second term above, the $s$-th component  of the contraction is defined as
\begin{align}
\lp [H^r_4: (D_x \otimes (\rho H_2)]_{[2,3,4,5:1,2,3,4]} \rp _s &=\sum_{i,j,k,p} (H^r_4)_{sijkp}  \partial_{x_i}\partial_{x_j} \lp \rho(H_2)_{kp} \rp
\end{align}
and analogously for the other contractions; and where
\begin{align}
	H_2 &= H_2(\eta,\Omega,\Omega^\perp) = \int_{\S^{n-1}} (u\otimes u) \ga  \d u, \label{eq:H_2}\\
	H_2^r&=H^r_2(\eta,\Omega,\Omega^\perp) = \int_{\S^{n-1}} P_{\Omega^\perp} u \otimes (u\otimes u)  (u\cdot \Omega) \ga \d u, \label{eq:H_2r}\\
	H_4 &= H_4(\eta,\Omega,\Omega^\perp) = \int_{\S^{n-1}} (u\otimes u\otimes u \otimes u) \ga \d u, \label{eq:H_4}\\
	H_4^r&=H^r_4(\eta,\Omega,\Omega^\perp) = \int_{\S^{n-1}} P_{\Omega^\perp} u \otimes (u\otimes u\otimes u \otimes u)  (u\cdot \Omega) \ga \d u. \label{eq:H_4r}
\end{align}

We can now rewrite these terms in such a way that they only depend on $\Omega$, $\Omega^\perp$ and $\eta$.

\begin{lem}[Dimension $n=2$] In case of only $n=2$, we have that $H_2, H_2^r, H_4, H_4^r$ correspond to expressions \eqref{eq:H2d2}, \eqref{eq:H2rd2}, \eqref{eq:H4d2} and \eqref{eq:H4rd2}, respectively.    
\end{lem}
The proof of this lemma is based on the change of variables \eqref{eq:change_of_variables} for $n=2$, i.e., $u=\cos\theta \Omega + \sin\theta\Omega^\perp$, and the same type of argument as in the proof of Lemma \ref{lem:model_integrals}. Thus, we skip it here.

\medskip
We conclude this section by computing the values of $H_2, H_2^r, H_4, H_4^r$ in dimension $n\geq 3$:
\begin{proof}[Proof of Proposition \ref{lem:Hdimension3}]
    For this proof, we consider the decomposition of $u\in \mathbb{S}^{n-1}$ into 
    \begin{align*}
    	u= \upar + \upp, \qquad \upar:= P_{\O} u=(u\cdot\Omega) \Omega, \quad \upp:= P_{\O^\perp}u,
    \end{align*} and the change of variables given in \eqref{eq:change_of_variables} where we have that
  $$(u\cdot\Omega)= \cos\theta, \qquad P_{\O^\perp}u=\sin\theta v.$$
  
Then, the proof of this lemma closely  the proof of Lemma 4.1 in \cite{DM-A20}. There, the authors prove that,
    \begin{align} \label{eq:Aintegral}
         \mathcal{A}_{ij}:=\ints \phi(u\cdot \O) (\upp)_i(\upp)_j \d u = \frac{1}{n-1}\lp \ints \phi(u\cdot \O) (1-(u\cdot\O)^2) \d u \rp (\pp)_{ij} 
    \end{align}
     where $\phi=\phi(u\cdot \O)$ is a given function. The authors also consider
    \begin{align*}
    S_{ijkl} &:= \ints \phi(u\cdot \Omega) (\upp)_i (\upp)_j (\upp)_k (\upp)_l \, \d u.
    \end{align*}and we additionally define
    \begin{align}
       \label{eq:CS} 
       C_S:= \frac{1}{(n-1)(n+1)}\int_{\mathbb{S}^{n-1}} \phi(u\cdot \O) (1-(u\cdot \O)^2)^2 \d u.
    \end{align}
   Notice that the only non-zero terms correspond to
    \begin{align} \label{eq:Sdef}
        S_{iiii} = 3C_S\Gamma_{iiii}, \quad
        S_{iijj} =C_S \Gamma_{iijj}, \quad
        S_{ijij} = C_S\Gamma_{ijij}, \quad
        S_{ijji}  =C_S\Gamma_{ijji},
    \end{align}
    where $\Gamma$ is given in \eqref{eq:Gamma} and $i,j \in \{1,...,n\}$ such that $i \neq j$. With this we are ready to carry out the proof.

The equality for $H_2$ in \eqref{eq:H2inlemma} is obtained by proceeding similarly as in the proof of Lemma \ref{lem:model_integrals} and using \eqref{eq:v_squared}.

\medskip
Next, we look at $H_2^r$ to prove \eqref{eq:H2rinlemma}. We have that
\begin{align*}
    H^r_2 &=  \ints \upp \otimes (\upp \otimes \upar) (u\cdot \Omega) \GetaA \d u + \ints \upp \otimes (\upar \otimes \upp) (u\cdot \Omega) \GetaA \d u\\
    &=  \frac{d_{2,2}}{n-1} P_{\Omega^\perp} \O + R^r_2.
\end{align*}
with $d_{2,2}$ given in Equation \eqref{eq:def_dkl}.

In the first equality above, the integrands that are odd in $\upp$ vanish (analogously to what happens when performing the change of variables \eqref{eq:change_of_variables}: the terms that are odd in $v$ vanish). In the second equality, for the first integral we used \eqref{eq:v_squared} and for the second integral we define 
\begin{align*}
(R^r_2)_{ijk} &= \ints (\upp)_i (\upar)_j (\upp)_k (u\cdot\O) \GetaA \d u = \Omega_j\int (\upp)_i(\upp)_k(u\cdot\Omega)^2 \GetaA \d u \\
&= \Omega_j \mathcal{A}_{ik} = \Omega_j \frac{d_{2,2}}{n-1} (\pp)_{ik},
\end{align*}
where in the definition of $\mathcal{A}$ we used $\phi(u\cdot \Omega) = (u\cdot\Omega)^2 \GetaA$.

Now consider
\begin{align}
\lp [\pp \otimes \O \otimes \pp]_{:24} \rp_{ijk} &= \sum_{p=1}^n (\pp)_{ip}\O_j (\pp)_{pk} = \O_j (\pp)_{ik},
\end{align}
where the second equality is a straightforward computation. Therefore, we conclude that
$$R^r_2= \frac{d_{2,2}}{n-1}[\pp \otimes \O \otimes \pp]_{:24}.$$

We proceed by focusing on $H_4$ and expanding it:
\begin{align*}
    H_4 &= \ints (\upp \otimes \upp\otimes \upp\otimes \upp) \GetaA \d u + \ints (\upp \otimes \upp\otimes \upar \otimes \upar) \GetaA \d u\\
    &+ \ints (\upar \otimes \upar\otimes \upp \otimes \upp) \GetaA \d u + \ints (\upar \otimes \upp\otimes \upp \otimes \upar) \GetaA \d u\\
    &+ \ints (\upp \otimes \upar\otimes \upp \otimes \upar) \GetaA \d u + \ints (\upp \otimes \upar\otimes \upar \otimes \upp )\GetaA \d u\\
   &+ \ints (\upar \otimes \upp\otimes \upar \otimes \upp) \GetaA \d u + \ints (\upar \otimes \upar \otimes \upar \otimes \upar) \GetaA \d u\\
   &=: D_1+D_2+D_3+\hdots+D_8,
\end{align*}
where all the terms odd in $\upp$ vanished.
Considering the change of variables \eqref{eq:change_of_variables}, we deduce from \eqref{eq:v_4tensor} that
$$D_1= d_{0,4}\frac{1}{(n-1)(n+1)}\Gamma.$$
From \eqref{eq:v_squared}, \eqref{eq:Aintegral}  and proceeding as before,  we have that
$$D_2= \frac{d_{2,2}}{n-1} \pp \otimes \Omega\otimes \O, \qquad D_3= \frac{d_{2,2}}{n-1} \O \otimes \O \otimes \pp, \qquad D_4 = \frac{d_{2,2}}{n-1} \O \otimes \pp \otimes\O.$$
Furthermore, proceeding as in the computation for $R^r_2$ we obtain
$$D_5 = \frac{d_{2,2}}{n-1}[\pp\otimes\O\otimes\pp\otimes\O]_{:24}, \qquad D_6 = \frac{d_{2,2}}{n-1}[\pp \otimes\O\otimes\O\otimes\pp]_{:25},$$
$$D_7= \frac{d_{2,2}}{n-1}[\O\otimes\pp\otimes\O\otimes\pp]_{:35}.$$
By applying the change of variables \eqref{eq:change_of_variables} to the last term we directly have that
$$D_8= d_{4,0}\, \O\otimes\O\otimes\O\otimes\O$$

Finally, we split $H_4^r$ into
\begin{align*}
    H_4^r = D_9 + D_{10},
\end{align*}
with
\begin{align*}
    D_9 = &\ints \lp \upp\otimes (\upp \otimes \upp\otimes \upp\otimes \upar) (u\cdot \O) + \upp\otimes (\upp \otimes \upp\otimes \upar\otimes \upp) (u\cdot \O) \rp \GetaA \d u\\
  + &\ints \lp \upp\otimes (\upp \otimes \upar\otimes \upp\otimes \upp) (u\cdot \O) + \upp\otimes (\upar \otimes \upp\otimes \upp\otimes \upp) (u\cdot \O) \rp \GetaA \d u,
  \end{align*} and
  \begin{align*}
  D_{10} = &\ints \lp \upp\otimes (\upp \otimes \upar\otimes \upar\otimes \upar) (u\cdot \O) + \upp\otimes (\upar \otimes \upp\otimes \upar\otimes \upar) (u\cdot \O) \rp \GetaA \d u\\
  + &\ints \lp \upp\otimes (\upar \otimes \upar\otimes \upp\otimes \upar) (u\cdot \O) + \upp\otimes (\upar \otimes \upar\otimes \upar\otimes \upp) (u\cdot \O) \rp \GetaA \d u,
\end{align*}
where the terms odd in $\upp$ vanished. Proceeding as in the computation of $R^r_2$ we have that
\begin{multline*}
    D_{10} = \frac{d_{4,2}}{n-1}\Big(\pp\otimes \O \otimes\O\otimes\O
    + [\pp \otimes \O\otimes\pp\otimes\O\otimes \O]_{:24}\\
    \qquad+ [\pp\otimes\O\otimes\O\otimes\pp\otimes\O]_{:25} + [\pp \otimes \O\otimes\O\otimes\O\otimes\pp]_{:26}\Big).
\end{multline*}
To compute $D_{9}$, we consider the second component in the first integral as an example. This integral is a 5-tensor. We apply $\upar=(u\cdot \O)\O$ and look at the components $i,j,k,l,m$
\begin{multline*}
    \lp \ints \upp\otimes (\upp \otimes \upp\otimes \O\otimes \upp) (u\cdot \O)^2\GetaA \d u  \rp_{ijklm}\\
     = \ints (\upp)_i(\upp)_j(\upp)_k\O_l (\upp)_m (u\cdot\O)^2\GetaA \d u =\O_l S_{ijkm},
\end{multline*}
where in the definition of $S_{ijkm}$ we have  $\phi(u\cdot\O)= (u\cdot \O)^2\GetaA$, which in this case means that $C_S = d_{2,4}$.
Putting all the terms of $D_{9}$ together and by considering all the combinations of indices we have that
$$D_{9}= \frac{d_{2,4}}{(n-1)(n+1)} T,$$
with $T$ given in \eqref{eq:defT}. This completes the proof. 
\end{proof}

 \section{Conclusions and open questions}
 \label{sec:conclusions}
 
 In this article, we have explored the effects of an anisotropic repulsive potential on inert particles. We derived both kinetic \eqref{eq:dimless_kinetic_eq} and macroscopic equations \eqref{eq:macro_rho}-\eqref{eq:macro_Omega} and discussed potential interpretations of the latter. The equation for the particle density $\rho$ is independent of the particles' mean-nematic direction $\Omega$. However, the anisotropy of the particles slows down the non-linear diffusion of the particle density $\rho$. In contrast, the equation for $\Omega$ is more complex and challenging to interpret. It consists of transport and diffusion terms, resulting from the effect of the repulsive potential on the particles' position. Additionally, the repulsive potential acting on the particles' direction creates a complex, diffusive-type operator with different signs for oblate and prolate particles.

 Many models for collective dynamics impose an alignment force directly on particle directions. Our main goal in this article was to observe the nematic alignment and the spatial effects of the anisotropic Gaussian-type repulsive potential on the particle dynamics without imposing the alignment directly. 
 
 However, when we derived this effect directly from an anisotropic repulsive potential, we observed intriguing phenomena. For instance, we can understand how particle anisotropy influences the evolution of particle density, how particle positions affect their directions, and what the impact of the specific interaction potential considered on the directions is (through the term $B_f$).

 A numerical study was carried out to test some of the analytical results obtained, namely: (i) the global alignment breaks down in the regime where the macroscopic equation becomes ill-posed, as characterized by the well-posedness condition \eqref{eq:intro_condition_wellposedness}; (ii)
 particles close to be spherical or low particle densities result in slower rates of alignment or a complete lack thereof; (iii) lastly, we also confirmed numerically that the Berne-Pechukas potential does not lead to global alignment at the particle level.

 The well-posedness of the continuum equations will be the object of future studies, along with the consideration of other types of anisotropic repulsive potentials, such as generalizations of the Lennard-Jones potential. Apart from this, there are three outstanding open problems left out in this work, namely, the well-posedness of the kinetic equation, showing that Assumption \ref{as:A} holds, and obtaining an explicit lower bound for the operator $K$ \eqref{eq:Ceta_coeff_porousmedium}. Another interesting aspect for future research is the study of compactly supported potentials, which may better represent contact interactions.

\section*{Acknowledgements}

The authors would like to thank Pierre Degond for wonderful and fruitful discussions.

\medskip

\noindent The work of SMA and CW was funded in part by the Austrian Science Fund (FWF) project \href{https://doi.org/10.55776/F65}{10.55776/F65} and in part by the Vienna Science  and  Technology  Fund  (WWTF)  [10.47379/VRG17014].
CW was funded partly by the Austrian Science Fund (FWF) {W1261-B28. SP is supported by Japan Society for the Promotion of Science (JSPS) KAKENHI (grant number JP24K16962). HY is supported by the Dutch Research Council (NWO) under the NWO-Talent Programme Veni ENW project MetaMathBio with the project number VI.Veni.222.288. This research was completed while HY was visiting the Okinawa Institute of Science and Technology (OIST) through the Theoretical Sciences Visiting Program (TSVP). SP gratefully acknowledges the hospitality of the TSVP during his stay at OIST. 

\medskip 
\noindent For the purpose of open access, the authors have
applied a Creative Commons Attribution (CC-BY) licence to any Author Accepted Manuscript version
arising from this submission. 

\appendix

\section{Appendices}

\subsection{Proof of Lemma \ref{lem:rescaling the potential}}
\label{appendix_lem rescaling the potential}

\begin{proof}
	Considering the scaling $\ell = \eps \ell'$ and $d = \eps d'$ and a change of variable  $z = \frac{x_2-x_1}{\eps} $, so $\eps^{n} \d z = \d x_2$, we obtain
	\begin{align*}
		V_f (t, x_1, u_1) &= \ints \intx  V_{b_\text{WG}} \left ( u_1, u_2, \frac{x_2-x_1}{\eps} \right ) f^\eps(t, x_2, u_2) \d x_2 \d u_2 \\
		&= \eps^{n} \ints \intx  V_{b_\text{WG}} \left ( u_1, u_2, z \right ) f^\eps(t, x_1 + \eps z, u_2) \d z \d u_2.
	\end{align*} 
	Now, we perform a Taylor expansion of $f^\eps(t,\cdot,u_2)$ at $x_1$ to obtain 
	\begin{align*}
		f^\eps(t, x_1 + \eps z, u_2) = f^\eps(t, x_1,  u_2) + \eps z\cdot \nabla_x f^\eps(t, x_1, u_2)  + \frac{1}{2}  \eps^{2} z^T\nabla_x^2 f^\eps(t, x_1, u_2)\cdot z  + \mathcal O (\eps^{3}).
	\end{align*}
	Since, we consider the weighted Gaussian potential $V_{b_{\text{WG}}}$, as defined in Equation \eqref{eq:weighted_Gaussian_potential}, we use \linebreak $\intx z V_{b_{\text{WG}}}(u_1, u_2, z) \d z =0$, and thus for the potential $V_f(t, x_1,u_1)$, we have that 
	\begin{align*}
		V_f (t, x_1, u_1) =  &\, \eps^{n} \ints   \left ( \intx V_{b_{\text{WG}}} \left ( u_1, u_2, z \right ) \d z \right ) f^\eps(t,x_1,  u_2) \d u_2 \\
		& +\frac{\eps^{n+2}}{2} \ints \intx  V_{b_{\text{WG}}}( u_1, u_2,z) z^T \nabla^2_x f^\eps(t, x_1, u_2)\cdot z \d z \d u_2 + \mc O(\eps^{n+3}). 
	\end{align*} 
Notice that with the scaling factor \eqref{eq:scaling_factor_WG} we have for the potential $V_{b_\text{WG}}$,
\begin{align} \label{eq:prop_b_WG}
	\begin{split}
		\intx V_{b_{\text{WG}}}(u_1, u_2,z) \d z &= b_{\text{WG}}^{2}(u_1, u_2), \\
		\intx (z\otimes z) V_{b_{\text{WG}}}(u_1,u_2,z) \d z &= \frac{1}{2}b_{\text{WG}}^{2}(u_1, u_2)\Sigma.
	\end{split}
\end{align}
Using \eqref{eq:prop_b_WG}, we obtain
	\begin{align*}
		V_f (t,x_1, u_1) = & \, \eps^n \ints \left ( 1- \chi^2 (u_1 \cdot u_2)^2\right )f(t,x_1, u_2) \d u_2 \\
		 &+ \frac{\eps^{n+2}}{4} \ints\left ( 1- \chi^2 (u_1 \cdot u_2)^2\right ) \Sigma (u_1,u_2): \nabla^2_x f(t,x_1, u_2) \d u_2 + \mc O (\eps^{n+3}). 
	\end{align*}   
Then,
\begin{align*}
	V_f^\eps (t,x_1, u_1) = \frac{1}{\eps^n} V_f (t,x_1, u_1) = W_f (t,x_1, u_1) + \eps^2 B_f (t,x_1, u_1) +\mc O (\eps^3).
\end{align*} where $W_f$ and $B_f$ are Equations \eqref{eq:W_f} and \eqref{eq:B_f} in the lemma.
\end{proof} 

\subsection{Proof of Lemma \ref{lem:model_integrals}}
\label{appendix_lem_model_integrals}

\begin{proof}

We will consider the change of variables given in \eqref{eq:change_of_variables} throughout the following proof and thus we can recast
\begin{align*}
    \GCI &= h_\eta(\cos\theta) \sin\theta \, v,\\
    \GetaA(u) &=\Getatheta,
\end{align*}
where $\Getatheta$ is given in \eqref{lem:var_trafo_G_h}.

\medskip
To begin with, doing the change of variables \eqref{eq:change_of_variables}, the integral \eqref{eq:intG} corresponds to
\begin{align*}
    \ints (u\cdot \Omega)^k \, \GetaA \GCI\, \d u & =c_{k,1} \lp\int_{S^{n-2}} v \, \d v \rp.
\end{align*}
The integral in $v$ is zero since the integrand is odd. Moreover, the function $c_{k,p}$ is defined in \eqref{eq:def_ckl}.

Proceeding similarly, we consider \eqref{eq:v_squared} and transform \eqref{eq:intGw} into  
\begin{align*}
    \ints (u\cdot w)(u\cdot \Omega)^k \, \GetaA \GCI \, \d u &= c_{k,2} \lp \int_{\S^{n-2}}v\otimes v\, \d v\rp w= 
	\frac{c_{k,2}}{n-1}P_{\Omega^\perp} w.
\end{align*}

Now, the integral \eqref{eq:intGpartial} is computed analogously, since we know the value of $\partial\GetaA$ from Lemma \ref{lem:derivative_GetA} and using that
$P_{\Omega^\perp}\partial \Omega=\partial\Omega$.

Integral \eqref{eq:intGw_partial} is also computed analogously using relation \eqref{eq:v_odd}. 

Moreover, considering \eqref{eq:v_odd} we recast integral \eqref{eq:Gw2_squared} as
\begin{align*}
    \ints (u\cdot\Omega)^{2k} (u\cdot w)^2  \, \GetaA \GCI \, \d u = 2(\Omega\cdot w) c_{2k+1,1} \lp\int_{S^{n-2}}v\otimes v \, \d v\rp w,
\end{align*}
and since $w\perp \Omega$ the integral vanishes.

Finally, the integral \eqref{eq:intLaplacian} is computed in the same way as above using the corresponding expression for $\partial^2\GetaA$ provided in Lemma \ref{lem:derivative_GetA}.
\end{proof}
\subsection{Numerical approximation of \texorpdfstring{$K(\eta)$}{K(n)} and additional figures}

In the following, we outline the method we use to approximate the diffusion parameter $K(\eta)$ in Figures \ref{fig:K_bounds} and \ref{fig:K_bounds_2}. One can approximate $S_2(\eta)$ via applying numerical quadrature such as the adaptive Gauss–Kronrod quadrature to \eqref{eq:def:S_2_theta}, we refer to this approximation as $\tilde{S_2}$. Next, the function $\eta(\rho)$ can be interpolated by computing $\tilde S_2(\eta_j)$ at equidistant points $\eta_j = \frac{C (1 -\chi^2) j}{m}$ for $1 \leq j \leq m$ where $m$ is the number of interpolation points and $C > 0$ is a large enough constant. Using these points, one can define $\tilde{\eta}$ as the Akima interpolation of $ \lp \frac{\eta_j}{\alpha \tilde S_2(\eta_j)}, \eta_j \rp $. The regularization effect of the Akima interpolation is useful to counter numerical instabilities for small values of $\eta$ close to the unknown $\eta^*$. 
Finally, to approximate $K$, we evaluate
\begin{align*}
    \tilde{K}(\rho) \coloneqq 1 - \frac{\chi^2}{n} - \sigma \frac{n-1}{n} \tilde{S_2}(\tilde{\eta}(\rho)) \tilde{\eta}'(\rho)
\end{align*} 
where the derivative $\tilde{\eta}'$ is the derivative of the Akima interpolant.

We finish the appendices with some additional figures from our numerical experiments.

\begin{figure}[ht!] 
\centering
\begin{subfigure}[b]{0.32\textwidth}
         \centering
         \includegraphics[width=\linewidth]
         {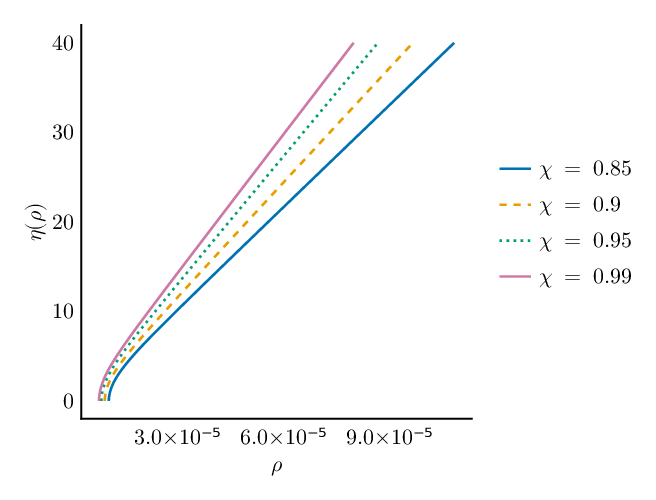}
         \caption{$\eta(\rho)$ with respect to $\rho$.}
         \label{fig:eta_and_S_2_a}
     \end{subfigure} 
\begin{subfigure}[b]{0.32\textwidth}
         \centering
         \includegraphics[width=\linewidth]
         {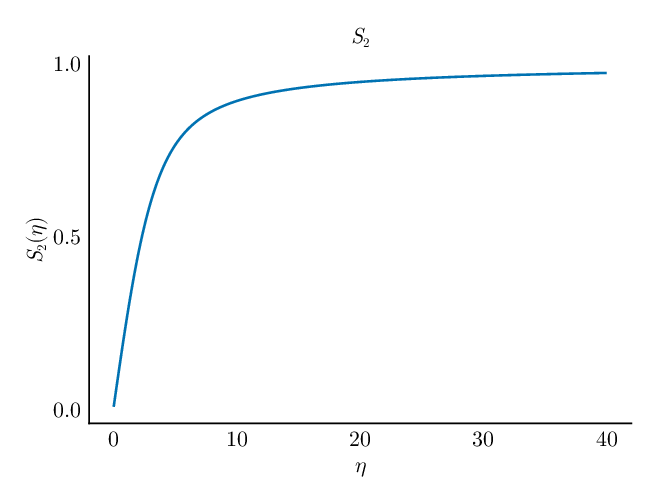}
         \caption{$S_2(\eta)$ with respect to $\eta$.}
         \label{fig:eta_and_S_2_b}
     \end{subfigure} 
\begin{subfigure}[b]{0.32\textwidth}
	\centering
	\includegraphics[width=\linewidth]
	{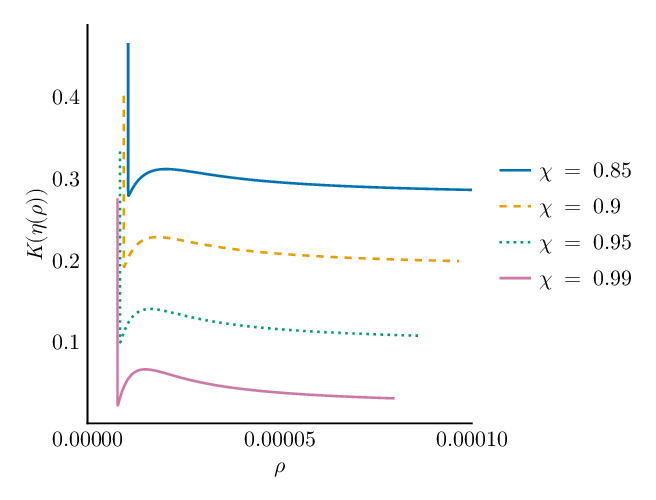}
	\caption{$K(\eta)$ with respect to $\rho$.}
	\label{fig:K_bounds_2}
\end{subfigure} 
\caption{(a) Plots of $\eta(\rho)$ with respect to $\rho$, for different $\chi$ values corresponding to different colors. Notice that $\eta$ is non-decreasing with respect to $\rho$. (b) A plot of $S_2(\eta)$ versus $\eta$. (c) Evolution of the diffusion coefficient $K$ with respect to the particle density $\rho$. Different colors correspond to different $\chi$ values.}
\label{fig:eta_and_S_2}
\end{figure}

\bibliography{Anisotropic_repulsion}

\end{document}